\documentclass[11pt,reqno]{amsart}
\usepackage{mathtools}
\usepackage{xcolor}
\usepackage[a4paper, left = 2cm, right = 2cm, top = 1.5cm, bottom = 1.5cm]{geometry}
\usepackage{dsfont} % used in \I
\usepackage[all]{xy}
\usepackage{graphicx} % for \scalebox in \op and for \resizebox
\usepackage[
hidelinks]{hyperref}
\usepackage[bb=boondox]{mathalfa}

\usepackage{tikz}  
\usetikzlibrary{shapes.geometric}

\usepackage{amssymb}

\usepackage{dsfont} % used in\I

\usepackage{amssymb,amsthm}
\usepackage{mathrsfs}
\usepackage{microtype}

\usepackage{mdframed}

\DeclareMathAlphabet{\mathbbold}{U}{bbold}{m}{n}

\newtheorem{proposition}{Proposition}[section]
\newtheorem{theorem}[proposition]{Theorem}
\newtheorem{lemma}[proposition]{Lemma}
\newtheorem{corollary}[proposition]{Corollary}
\newtheorem{definition}[proposition]{Definition}
\newtheorem{remark}[proposition]{Remark}
\newtheorem{question}[proposition]{Question}

\newcommand{\cst}{\ensuremath{\mathrm{C}^*}}

\newcommand{\id}{\mathrm{id}}
\newcommand{\ph}{\varphi}
\newcommand{\I}{\mathds{1}}
\newcommand{\comp}{\circ}
\newcommand{\tens}{\otimes}
\newcommand{\is}[2]{{\left\langle{#1}\,\vline\,#2\right\rangle}}
\newcommand{\ket}[1]{{\left|#1\right\rangle}}
\newcommand{\flip}{\boldsymbol{\chi}}

\newcommand{\CC}{\mathbb{C}}
\newcommand{\GG}{\mathbb{G}}
\newcommand{\HH}{\mathbb{H}}
\newcommand{\GP}{\GG_{\scriptscriptstyle{\mathfrak{P}}}}
\newcommand{\rP}{\rho_{\scriptscriptstyle{\mathfrak{P}}}}

\newcommand{\sA}{\mathsf{A}}
\newcommand{\sJ}{\mathsf{J}}

\DeclareMathOperator{\Aut}{Aut}
\DeclareMathOperator{\QAut}{QAut}
\DeclareMathOperator{\B}{B}
\DeclareMathOperator{\C}{C}
\DeclareMathOperator{\M}{M}
\DeclareMathOperator{\K}{K}
\title{Quantum Mycielskians: symmetries, twin vertices and distinguishing labelings}

\author{Arkadiusz Bochniak}
\address{Max-Planck-Institut f{\"u}r Quantenoptik, Garching, Germany}
\address{Munich Center for Quantum Science and Technology, Munich, Germany}
\email{arkadiusz.bochniak@mpq.mpg.de}

\author{Igor Che{\l}stowski}
\address{University of Warsaw, Faculty of Physics, Department of Mathematical Methods in Physics, Poland}
\email{ic393643@okwf.fuw.edu.pl}

\author{Pawe{\l} Kasprzak}
\address{University of Warsaw, Faculty of Physics, Department of Mathematical Methods in Physics, Poland}
\email{pawel.kasprzak@fuw.edu.pl}

\author{Piotr M.~So{\l}tan}
\address{University of Warsaw, Faculty of Physics, Department of Mathematical Methods in Physics, Poland}
\email{piotr.soltan@fuw.edu.pl}

\keywords{Quantum graph, Mycielskian, quantum group, twin vertices, distinguishing number}

\makeatletter
\@namedef{subjclassname@2020}{\textup{2020} Mathematics Subject Classification}
\makeatother

\begin{document}
    \maketitle

\begin{abstract}
We investigate a quantum generalization of the Mycielski construction for quantum graphs. In particular, we analyze the quantum symmetries of quantum Mycielskians and their relation to the symmetries of the underlying quantum graphs. We introduce a quantum analogue of graphs with twin vertices and show that, for classical graphs with a small number of vertices, this notion reduces to the classical one. Finally, we propose a quantum extension of the distinguishing number and examine its behavior under the quantum Mycielski construction.
\end{abstract}

\section{Introduction}\label{sec:intr}

Classical graph theory has a wide range of intriguing applications beyond pure mathematics, with particularly strong connections to computer science and information theory. In communication theory, graphs are used to model channels, where vertices represent possible inputs and outputs, and edges encode confusability relation between signals. This perspective naturally extends to quantum information theory, where the role of confusability graphs is played by non-commutative operator systems associated with quantum channels (see, e.g.,~\cite{Brannan_2019, BEV, DuanSeveriniWinter2013, Ganesan, paupaandme}). In this setting, Kraus operators of a quantum channel generate an algebraic structure that serves as the quantum analogue of the adjacency relations of a graph. The passage from classical to quantum thus requires reformulating standard graph invariants in operator-algebraic terms. Among the classical quantities that admit such generalizations are chromatic numbers - capturing the minimal resources needed for proper labelings - and clique numbers, which quantify the largest completely connected substructures. The (classical) chromatic number of a graph is the minimum number of labels (colors) that can be assigned to its vertices, so that the endpoints of every edge receive distinct labels. The clique number, by contrast, is the maximum size of a complete graph that can be embedded as a subgraph of the original. For example, graphs with clique number $2$ are precisely the triangle-free graphs. These invariants, when reinterpreted in the non-commutative framework, provide valuable insights into zero-error communication capacities of quantum channels and the interplay between combinatorial and operator-algebraic methods.

A long-standing problem in classical graph theory was to determine whether triangle-free graphs with arbitrarily large chromatic number exist. Equivalently, one asked whether the absence of small cliques imposes an upper bound on the chromatic complexity of a graph. This problem was resolved in~\cite{Mycielski}, where Mycielski introduced an explicit iterative construction. Given a triangle-free graph $G$, he defined a new graph $\mu(G)$, now known as the {\it Mycielskian} of $G$, which is also triangle-free, contains $G$ as an induced subgraph, and whose chromatic number satisfies 
\[
\chi(\mu(G)) = \chi(G) + 1.
\]  
By iterating this construction, one obtains an infinite family of triangle-free graphs with unbounded chromatic number, thereby providing a negative answer to the aforementioned problem. The Mycielski construction has since become a paradigmatic example of how graph transformations can alter chromatic properties without affecting clique structure.  

The natural question of extending this phenomenon to the quantum setting was addressed in~\cite{BochKaspMyc}, where two of the present authors introduced a non-commutative analogue of Mycielski's construction. In this framework, classical graphs are replaced by {\it quantum graphs}, modeled by operator systems generated by Kraus operators of quantum channels. The proposed construction, termed the {\it quantum Mycielskian}, takes a finite quantum graph $\mathcal{G}$ as input and produces a new one that contains the original as a subgraph in the operator-algebraic sense. This quantum analogue preserves certain structural features of the classical case but exhibits important differences in the behavior of invariants. Specifically, it was shown that while the clique number remains unchanged under the transformation -- mirroring the classical situation -- the (quantum) chromatic number $\chi_\bullet$ does not necessarily increase by exactly one. Instead, it satisfies only a weaker property:
\[
\chi_{\bullet}(\mu(\mathcal{G})) \leq \chi_\bullet(\mathcal{G}) + 1,
\]
with strict equality not guaranteed in general.

In this paper, we investigate further properties of the Mycielski construction in the quantum framework, with an emphasis on how structural and symmetry-related features behave under this transformation. Section~\ref{subsec:class} begins with a review of the classical Mycielski construction, highlighting its role in generating triangle-free graphs of arbitrarily large chromatic number. Section~\ref{subsec:qgqg} then sets up the operator-algebraic formalism of quantum graphs and introduces the actions of quantum groups on them. These actions provide the natural symmetry framework in the non-commutative setting, and their behavior under graph transformations will play a central role in our analysis.

A key technical aspect of this study concerns actions that preserve partitions of unity, which, in the commutative case, correspond precisely to proper vertex colorings or labelings of graphs. We discuss this notion in Section~\ref{subsec:action}, and its relevance for quantum analogues of chromatic invariants will be revisited throughout the paper (see, in particular, Section~\ref{sec:dist}). At the end of Section~\ref{sec:intr}, we recall the definition of the quantum Mycielskian and prepare the ground for our new results.

Section~\ref{sec:iso} is devoted to isomorphisms of quantum Mycielskians. In Section~\ref{sec:symm}, we investigate how the Mycielski construction interacts with the quantum automorphism groups of graphs, showing that this process naturally gives rise to the notion of \emph{quantum twin vertices}. These are studied in detail in Sections~\ref{subsec:twin} and~\ref{subsec:small}, as well as in the Appendix.

Finally, in Section~\ref{sec:dist}, we introduce the \emph{quantum distinguishing number}, a non-commutative analogue of the classical distinguishing number that measures the minimal labeling required to break all symmetries of a graph. We analyze how this invariant behaves under the quantum Mycielski construction and compare the results with the classical situation, further illustrating both the parallels and the divergences between the two frameworks.

\subsection{Classical Mycielski construction}\label{subsec:class}\hspace*{\fill}

Let $G=(V_G, E_G)$ be a (finite, undirected) graph. The \emph{Mycielskian} of $G$, denoted by $\mu(G)$ is a graph with the set of vertices $V_{\mu(G)}=\{\bullet\}\sqcup V_G\sqcup V_G$, where $\bullet$ denotes a distinguished point, called \emph{master vertex}. The edge set $E_{\mu(G)}$ can be specified by the adjacency matrix $A_{\scriptscriptstyle\mu(G)}$. We recall that the adjacency matrix $A_{\scriptscriptstyle G}$ for an undirected graph $G$ is a $\{0,1\}$-matrix whose columns and rows are indexed by the set of vertices, and $(A_{\scriptscriptstyle G})_{i,j}$ is non-zero if and only if $\{v_i,v_j\}\in E$. For the Mycielskian, we have
\[
A_{\scriptscriptstyle\mu(G)}=\begin{bmatrix}
    0 &\vec{0}^T &\vec{1}^T\\
   \vec{0}& A_{\scriptscriptstyle G} & A_{\scriptscriptstyle G}\\
   \vec{1}& A_{\scriptscriptstyle G} & 0
\end{bmatrix},
\]
where $\vec{1}$ (resp. $\vec{0}$) is the column vector with all entries being $1$ (resp. $0$). Mycielski's theorem~\cite{Mycielski} says that denoting by $\omega(\cdot)$ the clique number of a graph we have $\omega(\mu(G))=\omega(G)$, i.e.~the clique number remains unchanged under the passage from $G$ to $\mu(G)$, and for the chromatic numbers $\chi$ we have $\chi(\mu(G))=\chi(G)+1$.

The Mycielski construction can be generalized to so-called \emph{$r$-Mycielskian} ($r\geq 2$) also known as \emph{higher cones} over graphs (\cite{Stiebitz, VanNgoc}). In this case, 
\[
V_{\mu_{r-1}(G)}=\{\bullet\}\sqcup V^1\sqcup\ldots\sqcup V^r
\]
with $V^j=\{v_1^j,\ldots,v^j_{|V_G|}\}$, $j=1,\ldots,r$ denoting distinct copies of $V_G$, and
\[
E = E_G\cup\bigcup_{i=0}^{r-1}\bigl\{\{v^i_j,v^{i+1}_{j'}\}\,\bigr|\bigl.\,\{v_j,v_{j'}\}\in E_G\bigr\}\cup\bigl\{\{v^r_1,\bullet\},\{v^r_2,\bullet\},\ldots,\{v^r_{|V_G|},\bullet\}\bigr\}.
\]

\subsection{Quantum graphs and their quantum symmetries}\label{subsec:qgqg}\hspace*{\fill}

\begin{definition}
Let $\mathsf{A}$ be a finite-dimensional \cst-algebra. A \emph{$\delta$-form} on $\mathsf{A}$ is a faithful state $\psi$ on $\mathsf{A}$ such that writing $L^2(\mathsf{A},\psi)$ for the GNS Hilbert space for $\psi$ and regarding the multiplication map $m\colon\mathsf{A}\tens\mathsf{A}\to\mathsf{A}$ as a linear operator $L^2(\mathsf{A},\psi)\tens L^2(\mathsf{A},\psi)\to L^2(\mathsf{A},\psi)$ we have
\[
mm^*=\delta(\psi)^2\id_{L^2(\mathsf{A},\psi)}
\]
for some positive constant $\delta(\psi)$. Since $\mathsf{A}$ and $L^2(\mathsf{A},\psi)$ are naturally isomorphic as vector spaces, we usually write $\id_{\mathsf{A}}$ instead of $\id_{L^2(\mathsf{A},\psi)}$ even when viewed as a map of Hilbert spaces.
\end{definition}

Let $\mathfrak{C}$ be the category of $\cst$-algebras as defined in \cite{Woronowicz79,Woronowicz95}, and $\mathfrak{C}^{\mathrm{op}}$ its dual category. A \emph{quantum space} $\mathbb{X}$ is an object from $\mathfrak{C}^{\mathrm{op}}$ and its corresponding $\cst$-algebra is denoted by $\C(\mathbb{X})$. We say that a quantum space $\mathbb{X}$ is \emph{finite} if the \cst-algebra $\C(\mathbb{X})$ is finite-dimensional.

A \emph{quantum graph} $\mathcal{G}$ is specified by elements of a triple $(\mathbb{V}_{\scriptscriptstyle\mathcal{G}},\psi_{\scriptscriptstyle\mathcal{G}},A_{\scriptscriptstyle\mathcal{G}})$, where
\begin{itemize}
\item $\mathbb{V}_{\scriptscriptstyle\mathcal{G}}$ is a finite quantum space,
\item $\psi_{\scriptscriptstyle\mathcal{G}}$ is a $\delta$-form on the \cst-algebra $\C(\mathbb{V}_{\scriptscriptstyle\mathcal{G}})$,
\item $A_{\scriptscriptstyle\mathcal{G}}$ is a self-adjoint operator on the GNS Hilbert space $L^2(\C(\mathbb{V}_{\scriptscriptstyle\mathcal{G}}),\psi_{\scriptscriptstyle\mathcal{G}})$ such that
\[
m(A\tens A)m^*=A
\quad\text{and}\quad
A=(\id_{\C(\mathbb{V}_{\scriptscriptstyle\mathcal{G}})}\tens\eta^* m)(\id_{\C(\mathbb{V}_{\scriptscriptstyle\mathcal{G}})}\tens  A\tens\id_{\C(\mathbb{V}_{\scriptscriptstyle\mathcal{G}})})(m^*\eta\tens\id_{\C(\mathbb{V}_{\scriptscriptstyle\mathcal{G}})}),
\]
where
\[
m\colon\C(\mathbb{V}_{\scriptscriptstyle\mathcal{G}})\tens\C(\mathbb{V}_{\scriptscriptstyle\mathcal{G}})\longrightarrow\C(\mathbb{V}_{\scriptscriptstyle\mathcal{G}})
\quad\text{and}\quad
\eta\colon\mathbb{C}\longrightarrow\C(\mathbb{V}_{\scriptscriptstyle\mathcal{G}})
\]
are the multiplication and unit maps regarded as linear operators between Hilbert spaces $L^2(\C(\mathbb{V}_{\scriptscriptstyle\mathcal{G}}),\psi_{\scriptscriptstyle\mathcal{G}})\tens L^2(\C(\mathbb{V}_{\scriptscriptstyle\mathcal{G}}),\psi_{\scriptscriptstyle\mathcal{G}})\to L^2(\C(\mathbb{V}_{\scriptscriptstyle\mathcal{G}}),\psi_{\scriptscriptstyle\mathcal{G}})$ and $\CC\to L^2(\C(\mathbb{V}_{\scriptscriptstyle\mathcal{G}}),\psi_{\scriptscriptstyle\mathcal{G}})$ respectively.
\end{itemize}

For alternative approaches to quantum graphs and comparison between them see \cite{Daws} and references therein. In what follows, the \cst-algebra $\C(\mathbb{V}_{\scriptscriptstyle\mathcal{G}})$ will be denoted by $\C(\mathcal{G})$ and the Hilbert space $L^2(\C(\mathbb{V}_{\scriptscriptstyle\mathcal{G}}),\psi_{\scriptscriptstyle\mathcal{G}})$ by $L^2(\mathcal{G})$. We will also shorten $\delta(\psi_{\scriptscriptstyle\mathcal{G}})$ to $\delta_{\scriptscriptstyle\mathcal{G}}$. A quantum graph $\mathcal{G}$ is called {\it irreflexive} if $m(A\otimes \id_{\C(\mathbb{V}_{\scriptscriptstyle\mathcal{G}})})m^\ast=0$.

In what follows, we will freely identify the \cst-algebras of functions on the quantum sets of vertices of the quantum graph with the Hilbert spaces obtained via the GNS construction for the corresponding $\delta$-forms. In particular, all linear maps between the considered \cst-algebras (homomorphisms and others) may be understood as maps of Hilbert spaces, and we will have at our disposal their hermitian adjoints.

\begin{definition}[{\cite[Definition 6.3]{BEV}}]
Let $\mathcal{G}$ be a quantum graph. The \emph{quantum symmetry group} of $\mathcal{G}$ is the compact quantum group $\QAut(\mathcal{G})$ equipped with a $\psi_{\scriptscriptstyle\mathcal{G}}$-preserving action $\rho_{\scriptscriptstyle\mathcal{G}}\colon\C(\mathcal{G})\to\C(\mathcal{G})\tens\C(\QAut(\mathcal{G}))$ such that $\rho_{\scriptscriptstyle\mathcal{G}}\comp{A_{\scriptscriptstyle\mathcal{G}}}=(A_{\scriptscriptstyle\mathcal{G}}\tens\id_{\C(\QAut(\mathcal{G}))})\comp\rho_{\scriptscriptstyle\mathcal{G}}$ and which is universal among all $\psi_{\scriptscriptstyle\mathcal{G}}$-preserving compact quantum group actions $\alpha\colon\C(\mathcal{G})\to\C(\mathcal{G})\tens\C(\mathbb{G})$ such that $\alpha\comp A_{\scriptscriptstyle\mathcal{G}} = (A_{\scriptscriptstyle\mathcal{G}}\tens\id_{\C(\mathbb{G})})\comp\alpha$.
\end{definition}

Any choice of a basis $\{e_1,\dotsc,e_n\}$ in $\C(\mathcal{G})$ determines elements $u_{i,j}$ of $\C(\QAut(\mathcal{G}))$ by
\[
\rho_{\scriptscriptstyle\mathcal{G}}(e_j)=\sum_{i=1}^n e_i\tens u_{i,j}.
\]
If the basis $\{e_1,\dotsc,e_n\}$ is orthonormal in $L^2(\mathcal{G})$ then the matrix $u$ with elements $u_{i,j}$ is unitary. In all cases it is a finite-dimensional representation of $\QAut(\mathcal{G})$ and $\C(\QAut(\mathcal{G}))$ is generated by $\{u_{i,j}\,|\,i,j=1,\dotsc,\dim{\C(\mathcal{G})}\}$. The compatibility of $\rho_{\scriptscriptstyle\mathcal{G}}$ with $A_{\scriptscriptstyle\mathcal{G}}$ is equivalent to $(\I\tens{A_{\scriptscriptstyle\mathcal{G}}})u=u(\I\tens{A_{\scriptscriptstyle\mathcal{G}}})$.

The quantum symmetry group of a quantum graph $\mathcal{G}$ is by construction a quantum subgroup of a well-studied quantum group $\Aut^+(\C(\mathbb{V}_{\scriptscriptstyle\mathcal{G}}),\psi_{\scriptscriptstyle\mathcal{G}})$ of symmetries of the quantum space $\mathbb{V}_{\scriptscriptstyle\mathcal{G}}$ together with the $\delta$-form $\psi_{\scriptscriptstyle\mathcal{G}}$ (the action of this larger quantum group does not preserve the adjacency matrix). The quantum groups of symmetries of finite quantum spaces equipped with a $\delta$-form were introduced in \cite{Banica02}. Finally, we note that in the existing literature the notion of an action of quantum groups is often referred to as a coaction of the corresponding \cst-algebras. For brevity, however, we adopt the former convention.

\subsection{Actions preserving a partition of unity}\label{subsec:action}\hspace*{\fill}

We now briefly introduce a construction not directly related to quantum graphs which will prove useful later. 

Let $\sA$ be a unital \cst-algebra and let $\rho\colon\sA\to\sA\tens\C(\GG)$ be an action of a compact quantum group $\GG$ on the quantum space underlying $\sA$. In the subsequent sections of this paper, $\sA$ will often take the form $\C(\mathbb{X})$ or $\C(\mathbb{X}) \otimes \B(\mathscr{H})$, for some Hilbert space $\mathscr{H}$, depending on the context. Let $\mathfrak{P}=\{P_1,\dotsc,P_n\}$ be a partition of unity in $\sA$. Let $\sJ$ be the ideal in $\C(\GG)$ generated by the elements
\[
\bigl\{(\omega\tens\id)\rho(P_a)-\omega(P_a)\I\,\bigr|\bigl.\,\omega\in\sA^*,\:a=1,\dotsc,n\bigr\}.
\]
Denote the quotient $\C(\GG)/\sJ$ by $\C(\GP)$ and let $\pi\colon\C(\GG)\to\C(\GP)$ be the corresponding quotient map. Furthermore, let $\rP=(\id\tens\pi)\comp\rho$. Then $\rP\colon\sA\to\sA\tens\C(\GP)$ and $\rP(P_a)=P_a\tens\I$ for all $a$ (since slice maps $(\omega\tens\id)$ separate elements of $\sA\tens\C(\GG)$).

Next, let $\mathbb{Y}$ be any compact quantum space and let $\ph\colon\C(\GG)\to\C(\mathbb{Y})$ be a unital $*$-homomorphism such that $(\id\tens\ph)\rho(P_a)=P_a\tens\I$ for all $a$. Then for any $\omega\in\sA^*$ and any $a$ we have
\[
0=(\omega\tens\id)\bigl((\id\tens\ph)\rho(P_a)-P_a\tens\I\bigr)=\ph\bigl((\omega\tens\id)\rho(P_a)-\omega(P_a)\I\bigr),
\]
so $\sJ\subset\ker\ph$. It follows that there exists a unique $\Lambda\colon\C(\GP)\to\C(\mathbb{Y})$ such that $\ph=\Lambda\comp\pi$.

Now take $\GP\times\GP$ for $\mathbb{Y}$ in the sense that $\C(\mathbb{Y})=\C(\GP)\tens\C(\GP)$ and let $\ph=(\pi\tens\pi)\comp\Delta_\GG$, where $\Delta_\GG$ is the comultiplication on $\C(\GG)$. Then by the reasoning above there exists a unique $\Delta_{\GP}\colon\C(\GP)\to\C(\GP)\tens\C(\GP)$ such that $(\pi\tens\pi)\comp\Delta_\GG=\Delta_{\GP}\comp\pi$. Moreover, we have
\[
(\rP\tens\id)\comp\rP=(\id\tens\ph)\comp\rho=(\id\tens\Delta_{\GP})\comp\rP.
\]
Of course
\[
\resizebox{\textwidth}{!}{\ensuremath{
\begin{aligned}
(\Delta_{\GP}\tens\id)\comp\Delta_{\GP}\comp\pi&=(\Delta_{\GP}\tens\id)\comp(\pi\tens\pi)\comp\Delta_\GG=(\pi\tens\pi\tens\pi)\comp(\Delta_\GG\tens\id)\comp\Delta_\GG\\
&=(\pi\tens\pi\tens\pi)\comp(\id\tens\Delta_\GG)\comp\Delta_\GG=(\id\tens\Delta_{\GP})\comp(\pi\tens\pi)\comp\Delta_\GG
=(\id\tens\Delta_{\GP})\comp\Delta_{\GP}\comp\pi,
\end{aligned}
}}
\]
so $\Delta_{\GP}$ is coassociative, since $\pi$ is surjective. Finally, applying $\pi\tens\pi$ to both sides of the equalities
\[
\overline{\bigl(\C(\GG)\tens\I\bigr)\Delta_\GG\bigl(\C(\GG)\bigr)}=\C(\GG)\tens\C(\GG)\quad\text{and}\quad
\overline{\Delta_\GG\bigl(\C(\GG)\bigr)\bigl(\I\tens\C(\GG)\bigr)}=\C(\GG)\tens\C(\GG)
\]
we find that the density conditions are satisfied for $(\C(\GP),\Delta_{\GP})$ and consequently $\GP$ is a compact quantum group.

\subsection{Mycielskians of quantum graphs}\label{sect:intro}\hspace*{\fill}

Here we recall the construction introduced in \cite{BochKaspMyc}. For $r\geq{2}$ the \emph{Mycielski construction of order $r-1$} for quantum graphs is the procedure of passing from a quantum graph $\mathcal{G}$ to a new quantum graph $\mu_{r-1}(\mathcal{G})$ defined as follows (\cite[Section 4]{BochKaspMyc}):
\begin{itemize}
\item the quantum space $\mathbb{V}_{\scriptscriptstyle\mu_{r-1}(\mathcal{G})}$ is the disjoint union of a point and $r$ copies of $\mathbb{V}_{\scriptscriptstyle\mathcal{G}}$: $\C(\mu_{r-1}(\mathcal{G}))=\CC\oplus\bigoplus\limits_{k=1}^r\C(\mathcal{G})$,
\item the state $\psi_{\scriptscriptstyle\mu_{r-1}(\mathcal{G})}$ is defined by the formula
\[
\psi_{\scriptscriptstyle\mu_{r-1}(\mathcal{G})}
\left(\left[
\begin{smallmatrix}
\lambda\\x_1\\[-2pt]\vdots\\[2pt]x_r
\end{smallmatrix}
\right]\right)=\tfrac{1}{1+r\delta_{\scriptscriptstyle\mathcal{G}}^2}\biggl(\lambda+\delta_{\scriptscriptstyle\mathcal{G}}^2\sum_{k=1}^r\psi_{\scriptscriptstyle\mathcal{G}}(x_i)\biggr)
\]
(we note  the obvious identification $L^2(\mu_{r-1}(\mathcal{G}))=\CC\oplus\bigoplus\limits_{k=1}^r L^2(\mathcal{G})$ as vector spaces),
\item the quantum adjacency matrix $A_{\scriptscriptstyle\mu_{r-1}(\mathcal{G})}$ acts on $L^2(\mu_{r-1}(\mathcal{G}))$ as
\[
A_{\scriptscriptstyle\mu_{r-1}(\mathcal{G})}
\left[\begin{smallmatrix}
\lambda\\x_1\\[-2pt]\vdots\\[2pt]x_r
\end{smallmatrix}\right]=
\left[\begin{smallmatrix}
\delta_{\scriptscriptstyle\mathcal{G}}^2\psi_{\scriptscriptstyle\mathcal{G}}(x^r)\\
A_{\scriptscriptstyle\mathcal{G}}(x_1+x_2)\\
A_{\scriptscriptstyle\mathcal{G}}(x_1+x_3)\\[-2pt]\vdots\\[2pt]
A_{\scriptscriptstyle\mathcal{G}}(x_{r-2}+x_{r})\\
\lambda\I_{\C(\mathcal{G})}+A_{\scriptscriptstyle\mathcal{G}}x_{r-1}
\end{smallmatrix}\right].
\]
\end{itemize}
Additionally, for $r=1$ we let $\mu_{r-1}(\mathcal{G})=\mu_0(\mathcal{G})$ be the original quantum graph $\mathcal{G}$. The quantum graph $\mu_{r-1}(\mathcal{G})$ contains $r$ ``copies'' of the original graph $\mathcal{G}$. In particular $\mu_1(\mathcal{G})$ also denoted $\mu(\mathcal{G})$ is a non-commutative generalization of the Mycielskian of a finite graph (\cite{Mycielski}).% this last remark should be shifted to the Introduction

We let $\iota_0\colon\CC\to L^2(\mu_{r-1}(\mathcal{G}))$ and $\iota_1,\dotsc,\iota_r\colon L^2(\mathcal{G})\to L^2(\mu_{r-1}(\mathcal{G}))$ denote the isometric embeddings onto the respective direct summands and $\iota_0^*\colon L^2(\mu_{r-1}(\mathcal{G}))\to\CC$ and $\iota_1^*,\dotsc,\iota_r^*\colon L^2(\mu_{r-1}(\mathcal{G}))\to L^2(\mathcal{G})$ are their hermitian adjoints:
\begin{align*}
\iota_0(\lambda)&=\sqrt{1+r\delta_{\scriptscriptstyle\mathcal{G}}^2}
\left[\begin{smallmatrix}
\lambda\\0\\[-5pt]\vdots\\0
\end{smallmatrix}\right],&
\iota_0^*
\left[\begin{smallmatrix}
\lambda\\x_1\\[-2pt]\vdots\\[2pt]x_r
\end{smallmatrix}\right]
&=\sqrt{\tfrac{1}{1+r\delta_{\scriptscriptstyle\mathcal{G}}^2}}\,\lambda,\\
\iota_k(x)
&=\sqrt{\tfrac{1+r\delta_{\scriptscriptstyle\mathcal{G}}^2}{\delta_{\scriptscriptstyle\mathcal{G}}^2}}
\left[\begin{smallmatrix}
0\\[-5pt]\vdots\\x\\[-5pt]\vdots\\0
\end{smallmatrix}\right],
&\iota_k^*
\left[\begin{smallmatrix}
\lambda\\x_1\\[-2pt]\vdots\\[2pt]x_r
\end{smallmatrix}\right]
&=\sqrt{\tfrac{\delta_{\scriptscriptstyle\mathcal{G}}^2}{1+r\delta_{\scriptscriptstyle\mathcal{G}}^2}}\,x_k
\end{align*}
for $k=1,\dotsc,r$. Note that although $\iota_0,\dotsc,\iota_r$, and their adjoints are not multiplicative, they are all $*$-maps. In this notation, we have (cf.~\cite[Eq.~(9)]{BochKaspMyc}):
\[
A_{\scriptscriptstyle\mu_{r-1}(\mathcal{G})}=\delta_{\scriptscriptstyle\mathcal{G}}\iota_r\eta_{\scriptscriptstyle\mathcal{G}}\iota_0^\ast +\delta_{\scriptscriptstyle\mathcal{G}}\iota_0\eta_{\scriptscriptstyle\mathcal{G}}^\ast\iota_r^\ast +\iota_1 A_{\scriptscriptstyle\mathcal{G}}\iota_1^\ast +\sum\limits_{k=1}^{r-1}\bigl(\iota_k A_{\scriptscriptstyle\mathcal{G}}\iota_{k+1}^\ast +\iota_{k+1} A_{\scriptscriptstyle\mathcal{G}}\iota_k^\ast\bigr).
\]

\section{Quantum isomorphisms of Mycielskians}\label{sec:iso}

Let $\mathcal{G}$ and $\mathcal{F}$ be quantum graphs. The notion of \emph{quantum isomorphism} of $\mathcal{G}$ and $\mathcal{F}$ has been studied in \cite[Definition 4.4, Corollary 4.8]{Brannan_2019} and \cite[Definition 6.5, Theorem 6.6]{BEV}. The definitions of quantum isomorphism of quantum graphs given in the two papers are equivalent (\cite[Theorem 6.6]{BEV}) and we will use both formulations.

According to \cite[Definition 6.5]{BEV} we say that $\mathcal{G}$ and $\mathcal{F}$ are \emph{quantum isomorphic} if there is a monoidal equivalence of the quantum automorphism groups $\Aut^+(\C(\mathbb{V}_{\scriptscriptstyle\mathcal{G}}),\psi_{\scriptscriptstyle\mathcal{G}})$ and $\Aut^+(\C(\mathbb{V}_{\scriptscriptstyle\mathcal{F}}),\psi_{\scriptscriptstyle\mathcal{F}})$ which maps the monoid object $\C(\mathbb{V}_{\scriptscriptstyle\mathcal{G}})$ to the monoid object $\C(\mathbb{V}_{\scriptscriptstyle\mathcal{F}})$ and maps the morphism $A_{\scriptscriptstyle\mathcal{G}}$ to $A_{\scriptscriptstyle\mathcal{F}}$ (the notion of monoidal equivalence of compact quantum groups is described e.g.~in \cite{Rijdt}).

The equivalent formulation of quantum isomorphism of $\mathcal{G}$ and $\mathcal{H}$ given in \cite[Section 4]{Brannan_2019} involves the $*$-algebra $\mathcal{O}(\QAut(\mathcal{F}),\QAut(\mathcal{G}))$ generated by elements $\{p_{ij}\}$ of a unitary matrix with the relations ensuring that upon a choice of orthonormal bases $\{e_j\}$ and $\{f_i\}$ of $\C(\mathbb{V}_{\scriptscriptstyle\mathcal{G}})$ and $\C(\mathbb{V}_{\scriptscriptstyle\mathcal{F}})$ the map
\[
\rho_{\scriptscriptstyle\mathcal{F},\mathcal{G}}\colon e_i\longmapsto\sum_{i}f_i\tens p_{ij}
\]
is a unital $*$-homomorphism $\C(\mathcal{G})\to\C(\mathcal{F})\tens\mathcal{O}(\QAut(\mathcal{F}),\QAut(\mathcal{G}))$ such that $\rho_{\scriptscriptstyle\mathcal{F},\mathcal{G}} A_{\scriptscriptstyle\mathcal{G}}=(A_{\scriptscriptstyle\mathcal{F}}\tens\id)\rho_{\scriptscriptstyle\mathcal{F},\mathcal{G}}$. The quantum graphs $\mathcal{G}$ and $\mathcal{F}$ are isomorphic if and only if the algebra $\mathcal{O}(\QAut(\mathcal{F}),\QAut(\mathcal{G}))$ is non-zero. Note that the matrix $p=[p_{ij}]_{i,j}$ can be treated as an element of $\B(L^2(\mathcal{G}),L^2(\mathcal{F}))\tens\mathcal{O}(\QAut(\mathcal{F}),\QAut(\mathcal{G}))$.

\begin{proposition}\label{prop:deltas}
Let $\mathcal{F}$ and $\mathcal{G}$ be quantum graphs. If $\mathcal{F}$ and $\mathcal{G}$ are quantum isomorphic, then $\delta_{\scriptscriptstyle\mathcal{F}}=\delta_{\scriptscriptstyle\mathcal{G}}$.
\end{proposition}

\begin{proof}
According to the first version of the definition of quantum isomorphism of $\mathcal{G}$ and $\mathcal{F}$ given above, the quantum groups $\Aut^+(\C(\mathbb{V}_{\scriptscriptstyle\mathcal{F}}),\psi_{\scriptscriptstyle\mathcal{F}})$ and $\Aut^+(\C(\mathbb{V}_{\scriptscriptstyle\mathcal{G}}),\psi_{\scriptscriptstyle\mathcal{G}})$ are monoidally equivalent. Hence the result follows from \cite[Theorem 4.7]{Rijdt}.
\end{proof}

\begin{proposition}%\label{prop:iso}
Let $\mathcal{G}$ and $\mathcal{F}$ be quantum isomorphic. Then $\mu_{r-1}(\mathcal{G})$ and $\mu_{r-1}(\mathcal{F})$ are quantum isomorphic for any $r\geq 1$.
\end{proposition}

\begin{proof}
Let
\[
p=[p_{ij}]_{i,j}\in\B\bigl(L^2(\mathcal{G}),L^2(\mathcal{F})\bigr)\tens\mathcal{O}\bigl(\QAut(\mathcal{F}),\QAut(\mathcal{G})\bigr)
\]
be the unitary matrix defining $\mathcal{O}(\QAut(\mathcal{F}),\QAut(\mathcal{G}))$ as above (remember the latter is non-zero) and let
\[
\rho_{\scriptscriptstyle\mathcal{F},\mathcal{G}}\colon\C(\mathcal{G})\ni e_j\longmapsto\sum\limits_i f_i\tens p_{ij}\in\C(\mathcal{F})\tens\mathcal{O}\bigl(\QAut(\mathcal{F}),\QAut(\mathcal{G})\bigr)
\]
be the corresponding unital $*$-homomorphism intertwining the quantum adjacency matrices. By Proposition~\ref{prop:deltas} we know that $\delta_{\scriptscriptstyle\mathcal{G}}=\delta_{\scriptscriptstyle\mathcal{F}}$ and for the remainder of the proof we will denote this value by $\delta$. The maps introduced above can be represented as
\[
p\colon L^2(\mathcal{G})\tens\mathcal{O}\bigl(\QAut(\mathcal{F}),\QAut(\mathcal{G})\bigr)\ni\xi\tens a\longmapsto\sum\limits_{i,j}\ket{f_i}\is{e_j}{\xi}\tens p_{ij}a\in L^2(\mathcal{F})\tens\mathcal{O}\bigl(\QAut(\mathcal{F}),\QAut(\mathcal{G})\bigr)
\]
and
\[
\rho_{\scriptscriptstyle\mathcal{F},\mathcal{G}}(\xi)=p(\xi\tens \I_{\mathcal{O}(\QAut(\mathcal{F}),\QAut(\mathcal{G}))}).
\]
We define an element $\widehat{p}\in\B(L^2(\mu_{r-1}(\mathcal{G})),L^2(\mu_{r-1}(\mathcal{F})))\tens\mathcal{O}(\QAut(\mathcal{F}),\QAut(\mathcal{G}))$:
\[
\widehat{p}=\varpi_{\scriptscriptstyle\mathcal{F},0}\varpi_{\scriptscriptstyle\mathcal{G},0}^*+\sum\limits_{k=1}^r\varpi_{\scriptscriptstyle\mathcal{F},k} p\varpi_{\scriptscriptstyle\mathcal{G},k}^*,
\]
where $\varpi_{\scriptscriptstyle\mathcal{G},j}=\iota_{\scriptscriptstyle\mathcal{G},j}\tens\I_{\mathcal{O}(\QAut(\mathcal{F}),\QAut(\mathcal{G}))}$ and $\varpi_{\scriptscriptstyle\mathcal{F},j}=\iota_{\scriptscriptstyle\mathcal{F},j}\tens\I_{\mathcal{O}(\QAut(\mathcal{F}),\QAut(\mathcal{G}))}$, with $\iota_{\scriptscriptstyle\mathcal{G},j}$ and $\iota_{\scriptscriptstyle\mathcal{F},j}$ ($j=1,\ldots,r$)denoting the isometric embeddings of $L^2(\mathcal{G})$ and $L^2(\mathcal{F})$ into $L^2(\mu_{r-1}(\mathcal{G}))$ and $L^2(\mu_{r-1}(\mathcal{F}))$, respectively.

First, notice that
\[
\begin{aligned}
\widehat{p}\widehat{p}^*&=\varpi_{\scriptscriptstyle\mathcal{F},0}\varpi_{\scriptscriptstyle\mathcal{F},0}^* +\sum\limits_{k=1}^r\varpi_{\scriptscriptstyle\mathcal{F},k} pp^*\varpi_{\scriptscriptstyle\mathcal{F},k}^* =\sum\limits_{j=0}^r\varpi_{\scriptscriptstyle\mathcal{F},j}\varpi_{\scriptscriptstyle\mathcal{F},j}^*=\id_{L^2(\mu_{r-1}(\mathcal{F}))},\\
\widehat{p}^*\widehat{p}&=\varpi_{\scriptscriptstyle\mathcal{G},0}\varpi_{\scriptscriptstyle\mathcal{G},0}^* +\sum\limits_{k=1}^r\varpi_{\scriptscriptstyle\mathcal{G},k} pp^*\varpi_{\scriptscriptstyle\mathcal{G},k}^* =\sum\limits_{j=0}^r\varpi_{\scriptscriptstyle\mathcal{G},j}\varpi_{\scriptscriptstyle\mathcal{G},j}^*=\id_{L^2(\mu_{r-1}(\mathcal{G}))},
\end{aligned}
\]
so that $\widehat{p}$ is unitary.

Next, we check the intertwining property. First, we compute
\begin{align*}
\widehat{p}A_{\scriptscriptstyle\mu_{r-1}(\mathcal{G})}&=\biggl(\varpi_{\scriptscriptstyle\mathcal{F},0}\varpi_{\scriptscriptstyle\mathcal{G},0}^* +\sum\limits_{k=1}^r\varpi_{\scriptscriptstyle\mathcal{F},k} p\varpi_{\scriptscriptstyle\mathcal{G},k}^*\biggr)\biggl(\delta\varpi_{\scriptscriptstyle\mathcal{G},r}\eta_{\scriptscriptstyle\mathcal{G}}\varpi_{\scriptscriptstyle\mathcal{G},0}^*+\delta\varpi_{\scriptscriptstyle\mathcal{G},0}\eta^*_{\scriptscriptstyle\mathcal{G}}\varpi_{\scriptscriptstyle\mathcal{G},r}^*+\varpi_{\scriptscriptstyle\mathcal{G},1} A_{\scriptscriptstyle\mathcal{G}}\varpi_{\scriptscriptstyle\mathcal{G},1}^*\\
&\qquad\qquad\qquad\qquad\qquad\qquad\qquad\qquad+\sum\limits_{k=1}^{r-1}\bigl(\varpi_{\scriptscriptstyle\mathcal{G},k} A_\mathcal{G}\varpi_{\scriptscriptstyle\mathcal{G},k+1}^* +\varpi_{\scriptscriptstyle\mathcal{G},k+1} A_\mathcal{G}\varpi_{\scriptscriptstyle\mathcal{G},k}^*\bigr)\biggr)\\
&=\delta\varpi_{\scriptscriptstyle\mathcal{F},0}\eta_{\scriptscriptstyle\mathcal{G}}^*\varpi_{\scriptscriptstyle\mathcal{G},r}^* +\delta\varpi_{\scriptscriptstyle\mathcal{F},r}p\eta_{\scriptscriptstyle\mathcal{G}}\varpi_{\scriptscriptstyle\mathcal{G},0}^* +\varpi_{\scriptscriptstyle\mathcal{F},1}pA_{\scriptscriptstyle\mathcal{G}}\varpi_{\scriptscriptstyle\mathcal{G},1}^*\\
&\qquad\qquad\qquad\qquad\qquad\qquad\qquad\qquad+\sum\limits_{j=1}^{r-1}\bigl(\varpi_{\scriptscriptstyle\mathcal{F},j}pA_{\scriptscriptstyle\mathcal{G}}\varpi_{\scriptscriptstyle\mathcal{G},j+1}^*+\varpi_{\scriptscriptstyle\mathcal{F},j+1}pA_{\scriptscriptstyle\mathcal{G}}\varpi_{\scriptscriptstyle\mathcal{G},j}^*\bigr),
\end{align*}
and similarly,
\begin{align*}
A_{\scriptscriptstyle\mu_{r-1}(\mathcal{F})}\widehat{p}=\delta\varpi_{\scriptscriptstyle\mathcal{F},0}\eta_{\scriptscriptstyle\mathcal{F}}^* p\varpi_{\scriptscriptstyle\mathcal{G},r}^* +\delta\varpi_{\scriptscriptstyle\mathcal{F},r}\eta_{\scriptscriptstyle\mathcal{F}}\varpi_{\scriptscriptstyle\mathcal{G},0}^* &+\varpi_{\scriptscriptstyle\mathcal{F},1}A_{\scriptscriptstyle\mathcal{F}}p\varpi_{\scriptscriptstyle\mathcal{G},1}^*\\
&+\sum\limits_{j=1}^{r-1}\bigl(\varpi_{\scriptscriptstyle\mathcal{F},j}A_{\scriptscriptstyle\mathcal{F}}p\varpi_{\scriptscriptstyle\mathcal{G},j+1}^*+\varpi_{\scriptscriptstyle\mathcal{F},j+1}A_{\scriptscriptstyle\mathcal{F}}p\varpi_{\scriptscriptstyle\mathcal{G},j}^*\bigr).
\end{align*}
Since $p\eta_{\scriptscriptstyle\mathcal{G}}=\eta_{\scriptscriptstyle\mathcal{F}}$, we also have $\eta_{\scriptscriptstyle\mathcal{G}}^*=\eta_{\scriptscriptstyle\mathcal{F}}^* p$. This shows that indeed the intertwining property holds. Finally, the map
\[
\rho_{\scriptscriptstyle\mu_{r-1}(\mathcal{F}),\mu_{r-1}(\mathcal{G})}
\left(\left[\begin{smallmatrix}
\lambda\\\xi_1\\[-3pt]\vdots\\[1pt]\xi_r
\end{smallmatrix}\right]\right)
=\widehat{p}
\left(\left[\begin{smallmatrix}
\lambda\\\xi_1\\[-3pt]\vdots\\[1pt]\xi_r
\end{smallmatrix}\right]\tens\I_{\mathcal{O}(\QAut(\mathcal{F}),\QAut(\mathcal{G}))}\right)
\]
is a unital $*$-homomorphism $\C(\mu_{r-1}(\mathcal{G}))\to
\C(\mu_{r-1}\mathcal{F}))\tens\mathcal{O}(\QAut(\mathcal{F}),\QAut(\mathcal{G}))$ 
by construction. Since the algebra $\mathcal{O}(\QAut(\mu_{r-1}(\mathcal{F})),\QAut(\mu_{r-1}(\mathcal{G})))$ is universal for such homomorphisms, it has a non-zero homomorphic image, so it is itself non-zero. It follows that $\mu_{r-1}(\mathcal{G})$ and $\mu_{r-1}(\mathcal{F})$ are quantum isomorphic.
\end{proof}

\section{Quantum symmetries of \texorpdfstring{$\mu_{r-1}(\mathcal{G})$}{mr(G)} from quantum symmetries of \texorpdfstring{$\mathcal{G}$}{G}}\label{sec:symm}

Our goal now is to determine how the Mycielski transformation affects the quantum symmetries of the original quantum graph. We start with the following:

\begin{proposition}\label{prop:Phir}
Let $\mathcal{G}$ be a quantum graph. Then $\QAut(\mathcal{G})$ acts faithfully on its quantum Mycielskians. More precisely, for $r\geq{2}$ define the linear map $\Phi_r\colon\C(\mu_{r-1}(\mathcal{G}))\to\C(\mu_{r-1}(\mathcal{G}))\tens\C(\QAut(\mathcal{G}))$ by
\[
\Phi_r(x)=\bigl(\iota_0\iota_0^*(x)\tens\I_{\C(\QAut(\mathcal{G}))}\bigr)+\sum_{k=1}^r(\iota_k\tens\id_{\C(\QAut(\mathcal{G}))})
\rho_{\scriptscriptstyle\mathcal{G}}\bigl(\iota_k^*(x)\bigr).
\]
Then $\Phi_r$ is a unital $*$-homomorphism and denoting by $\Delta$ the comultiplication on $\C(\QAut(\mathcal{G}))$ we have
\begin{subequations}
\begin{align}
(\Phi_r\tens\id_{\C(\QAut(\mathcal{G}))})\comp\Phi_r&=(\id_{\C(\QAut(\mathcal{G}))}\tens\Delta)\comp\Phi_r,\label{eq:PhirDelta}\\
\Phi_r\comp A_{\scriptscriptstyle\mu_{r-1}(\mathcal{G})} &= (A_{\scriptscriptstyle\mu_{r-1}(\mathcal{G})}\tens\id_{\C(\QAut(\mathcal{G}))})\comp\Phi_r,\label{eq:PhirA}
\end{align}
\end{subequations}
Moreover the set $\bigl\{(\omega\tens\id_{\C(\QAut(\mathcal{G}))})\Phi_r(x)\,\bigr|\bigl.\,x\in\C(\mu_{r-1}(\mathcal{G})),\:\omega\in\C(\mu_{r-1}(\mathcal{G}))^*\bigr\}$ generates the \mbox{\cst-algebra} $\C(\QAut(\mathcal{G}))$.
\end{proposition}

\begin{proof}
We first establish that $\Phi_r$ is a unital $*$-homomorphism. We have
\[
\I_{\C(\mu_{r-1}(\mathcal{G}))}=\sqrt{\tfrac{1}{1+r\delta_{\scriptscriptstyle\mathcal{G}}^2}}\biggl(\iota_0(1)+\delta_{\scriptscriptstyle\mathcal{G}}\sum_{k=1}^r\iota_k(\I_{\C(\mathcal{G})})\biggr).
\]
Hence
\begin{align*}
\Phi_r&(\I_{\C(\mu_{r-1}(\mathcal{G}))})=\sqrt{\tfrac{1}{1+r\delta_{\scriptscriptstyle\mathcal{G}}^2}}\biggl(\iota_0(1)\tens\I_{\C(\QAut(\mathcal{G}))}+\delta_{\scriptscriptstyle\mathcal{G}}\sum_{k,l=1}^r(\iota_k\tens\id_{\C(\QAut(\mathcal{G}))})\rho_{\scriptscriptstyle\mathcal{G}}\bigl(\iota_k^*\iota_l(\I_{\C(\mathcal{G})})\bigr)\biggr)\\
&\qquad=\sqrt{\tfrac{1}{1+r\delta_{\scriptscriptstyle\mathcal{G}}^2}}\biggl(\iota_0(1)\tens\I_{\C(\QAut(\mathcal{G}))}+\delta_{\scriptscriptstyle\mathcal{G}}\sum_{k=1}^r(\iota_k\tens\id_{\C(\QAut(\mathcal{G}))})(\I_{\C(\mathcal{G})}\tens\I_{\C(\QAut(\mathcal{G}))})\biggr)\\
&\qquad=\sqrt{\tfrac{1}{1+r\delta_{\scriptscriptstyle\mathcal{G}}^2}}\biggl(\iota_0(1)+\delta_{\scriptscriptstyle\mathcal{G}}\sum_{k=1}^r\iota_k(\I_{\C(\mathcal{G})})\biggr)\tens\I_{\C(\QAut(\mathcal{G}))}=\I_{\C(\mu_{r-1}(\mathcal{G}))}\tens\I_{\C(\QAut(\mathcal{G}))}.
\end{align*}

To prove the multiplicativity of $\Phi_r$ we temporarily denote by $m^{(0)}$, $m^{(1)}$, $m^{(2)}$ and $m^{(3)}$ the multiplication maps on $\CC$, $\C(\mathcal{G})$, $\C(\mu_{r-1}(\mathcal{G}))$ and $\C(\QAut(\mathcal{G}))$ respectively (the last one defined only on the algebraic tensor product) and recall from \cite[Section 4]{BochKaspMyc} that
\[
m^{(2)}=\sqrt{1+r\delta_{\scriptscriptstyle\mathcal{G}}^2}\biggl(\iota_0 m^{(0)} (\iota_0^*\tens\iota_0^*)+\delta_{\scriptscriptstyle\mathcal{G}}^{-1}\sum_{k=1}^r\iota_k m^{(1)} (\iota_k^*\tens\iota_k^*)\biggr).
\]
Then note that for $s\in\{1,\dotsc,r\}$ the map $\sqrt{\tfrac{\delta_{\scriptscriptstyle\mathcal{G}}^2}{1+r\delta_{\scriptscriptstyle\mathcal{G}}^2}}\iota_s$ is multiplicative, which means that
\begin{equation}\label{eq:mult-is} m^{(2)}\comp(\iota_s\tens\iota_s)=\sqrt{\tfrac{1+r\delta_{\scriptscriptstyle\mathcal{G}}^2}{\delta_{\scriptscriptstyle\mathcal{G}}^2}}\iota_s\comp m^{(1)}.
\end{equation}
Hence for $(p,q)\neq(0,0)$, with $\flip$ denoting the flip map, by \eqref{eq:mult-is} we have
\[
\resizebox{\textwidth}{!}{\ensuremath{
\begin{aligned}
\Phi_r&\comp m^{(2)}\bigl(\iota_p(a)\tens\iota_q(b)\bigr)=\Phi_r\left(\delta_{p,q}\sqrt{\tfrac{1+r\delta_{\scriptscriptstyle\mathcal{G}}^2}{\delta_{\scriptscriptstyle\mathcal{G}}^2}}\iota_p(ab)\right)\\
&=\delta_{p,q}\sqrt{\tfrac{1+r\delta_{\scriptscriptstyle\mathcal{G}}^2}{\delta_{\scriptscriptstyle\mathcal{G}}^2}}(\iota_p\tens\id_{\C(\QAut(\mathcal{G}))})\rho_{\scriptscriptstyle\mathcal{G}}(ab)\\
&=\delta_{p,q}\sqrt{\tfrac{1+r\delta_{\scriptscriptstyle\mathcal{G}}^2}{\delta_{\scriptscriptstyle\mathcal{G}}^2}}(\iota_p\tens\id_{\C(\QAut(\mathcal{G}))})
(m^{(1)}\tens m^{(3)})(\id_{\C(\mathcal{G})}\tens\flip\tens\id_{\C(\QAut(\mathcal{G}))})(\rho_{\scriptscriptstyle\mathcal{G}}\tens\rho_{\scriptscriptstyle\mathcal{G}})(a\tens b)\\
&=\delta_{p,q}
(m^{(2)}\tens m^{(3)})(\iota_p\tens\iota_p\tens\id_{\C(\QAut(\mathcal{G}))}\tens\id_{\C(\QAut(\mathcal{G}))})
(\id_{\C(\mathcal{G})}\tens\flip\tens\id_{\C(\QAut(\mathcal{G}))})
(\rho_{\scriptscriptstyle\mathcal{G}}\tens\rho_{\scriptscriptstyle\mathcal{G}})(a\tens b)\\
&=\delta_{p,q}
(m^{(2)}\tens m^{(3)})
(\id_{\C(\mu_{r-1}(\mathcal{G}))}\tens\flip\tens\id_{\C(\QAut(\mathcal{G}))})(\iota_p\tens\id_{\C(\QAut(\mathcal{G}))}\tens\iota_p\tens\id_{\C(\QAut(\mathcal{G}))})
(\rho_{\scriptscriptstyle\mathcal{G}}\tens\rho_{\scriptscriptstyle\mathcal{G}})(a\tens b)\\
&=(m^{(2)}\tens m^{(3)})
(\id_{\C(\mu_{r-1}(\mathcal{G}))}\tens\flip\tens\id_{\C(\QAut(\mathcal{G}))})
(\Phi_r\tens\Phi_r)\bigl(\iota_p(a)\tens\iota_q(b)\bigr)
\end{aligned}
}}
\]
Similarly we compute that $\Phi_r\comp m^{(2)}=(m^{(2)}\tens m^{(3)})
(\id_{\C(\mu_{r-1}(\mathcal{G}))}\tens\flip\tens\id_{\C(\QAut(\mathcal{G}))})
(\Phi_r\tens\Phi_r)$ on simple tensors of the form $\iota_0(\lambda)\tens\iota_0(\lambda')$ for any $\lambda,\lambda'\in\CC$. Since $\C(\mu_{r-1}(\mathcal{G}))\tens\C(\mu_{r-1}(\mathcal{G}))$ is spanned by elements of the form $\iota_k(a)\tens\iota_l(b)$ with $k,l\in\{0,\dotsc,r\}$, these computations show that $\Phi_r$ is multiplicative.

The fact that $\Phi_r$ is a $*$-map follows immediately from the remark in Section~\ref{sect:intro} that all maps involved in the definition of $\Phi_r$ are $*$-maps.

To check \eqref{eq:PhirDelta} we note that for $\lambda\in\CC$ we have $\Phi_r(\iota_0(\lambda))=\iota_0(\lambda)\tens\I_{\C(\QAut(\mathcal{G}))}$ which immediately shows that $(\Phi_r\tens\id_{\C(\QAut(\mathcal{G}))})\Phi_r(\iota_0(\lambda))=(\id_{\C(\mu_{r-1}(\mathcal{G}))}\tens\Delta)\Phi_r(\iota_0(\lambda))$ and similarly, since for any $k\in\{1,\dotsc,r\}$ we have $\Phi_r\comp\iota_k=(\iota_k\tens\id_{\C(\QAut(\mathcal{G}))})\comp\rho_{\scriptscriptstyle\mathcal{G}}$, one finds that
\begin{align*}
(\Phi_r\tens\id_{\C(\QAut(\mathcal{G}))})\Phi_r\bigl(\iota_k(a)\bigr)
&=(\Phi_r\tens\id_{\C(\QAut(\mathcal{G}))})(\iota_k\tens\id_{\C(\QAut(\mathcal{G}))})\rho_{\scriptscriptstyle\mathcal{G}}(a)\\
&=(\iota_k\tens\id_{\C(\QAut(\mathcal{G}))}\tens\id_{\C(\QAut(\mathcal{G}))})(\rho_{\scriptscriptstyle\mathcal{G}}\tens\id_{\C(\QAut(\mathcal{G}))})\rho_{\scriptscriptstyle\mathcal{G}}(a)\\
&=(\iota_k\tens\id_{\C(\QAut(\mathcal{G}))}\tens\id_{\C(\QAut(\mathcal{G}))})(\id_{\C(\mathcal{G})}\tens\Delta)\rho_{\scriptscriptstyle\mathcal{G}}(a)\\
&=(\id_{\C(\mu_{r-1}(\mathcal{G}))}\tens\Delta)(\iota_k\tens\id_{\C(\QAut(\mathcal{G}))})\rho_{\scriptscriptstyle\mathcal{G}}(a)\\
&=(\id_{\C(\mu_{r-1}(\mathcal{G}))}\tens\Delta)\Phi_r\bigl(\iota_k(a)\bigr)
\end{align*}
for any $a\in\C(\mathcal{G})$.

We now move on to the proof of \eqref{eq:PhirA}. To do this, we recall:
\[
A_{\scriptscriptstyle\mu_{r-1}(\mathcal{G})}=\delta_{\scriptscriptstyle\mathcal{G}}\iota_r\eta_{\scriptscriptstyle\mathcal{G}}\iota_0^*
+\delta_{\scriptscriptstyle\mathcal{G}}\iota_0\eta_{\scriptscriptstyle\mathcal{G}}^*\iota_r^*+\iota_1A_{\scriptscriptstyle\mathcal{G}}\iota_1^*
+\sum_{k=1}^{r-1}(\iota_kA_{\scriptscriptstyle\mathcal{G}}\iota_{k+1}^*+\iota_{k+1}A_{\scriptscriptstyle\mathcal{G}}\iota_k^*),
\]
where $\eta_{\scriptscriptstyle\mathcal{G}}$ is the unit map $\CC\to\C(\mathcal{G})$. It follows from this equation that
\begin{align*}
A_{\scriptscriptstyle\mu_{r-1}(\mathcal{G})}\iota_0&=\delta_{\scriptscriptstyle\mathcal{G}}\iota_r\eta_{\scriptscriptstyle\mathcal{G}},\\
A_{\scriptscriptstyle\mu_{r-1}(\mathcal{G})}\iota_1&=\iota_1A_{\scriptscriptstyle\mathcal{G}}+\iota_2A_{\scriptscriptstyle\mathcal{G}},\\
A_{\scriptscriptstyle\mu_{r-1}(\mathcal{G})}\iota_l&=\iota_{l-1}A_{\scriptscriptstyle\mathcal{G}}+\iota_{l+1}A_{\scriptscriptstyle\mathcal{G}},\qquad l=2,\dotsc,r-1,\\
A_{\scriptscriptstyle\mu_{r-1}(\mathcal{G})}\iota_r&=\delta_{\scriptscriptstyle\mathcal{G}}\iota_0\eta_{\scriptscriptstyle\mathcal{G}}^*+\iota_{r-1}A_{\scriptscriptstyle\mathcal{G}}.
\end{align*}
Using this we check that for $\lambda\in\CC$
\begin{align*}
(A_{\scriptscriptstyle\mu_{r-1}(\mathcal{G})}\tens\id_{\C(\QAut(\mathcal{G}))})
&\Phi_r\bigl(\iota_0(\lambda)\bigr)
=A_{\scriptscriptstyle\mu_{r-1}(\mathcal{G})}\iota_0(\lambda)\tens\I_{\C(\QAut(\mathcal{G}))}
=\delta_{\scriptscriptstyle\mathcal{G}}
\iota_r\bigl(\eta_{\scriptscriptstyle\mathcal{G}}(\lambda)\bigr)\tens\I_{\C(\QAut(\mathcal{G}))}\\
&=\delta_{\scriptscriptstyle\mathcal{G}}\lambda
\iota_r(\I_{\C(\mathcal{G})})\tens\I_{\C(\QAut(\mathcal{G}))}
=\delta_{\scriptscriptstyle\mathcal{G}}\lambda
(\iota_r\tens\id_{\C(\QAut(\mathcal{G}))})\rho_{\scriptscriptstyle\mathcal{G}}(\I_{\C(\mathcal{G})})\\
&=\delta_{\scriptscriptstyle\mathcal{G}}\lambda\Phi_r
\bigl(\iota_r(\I_{\C(\mathcal{G})})\bigr)
=\delta_{\scriptscriptstyle\mathcal{G}}\Phi_r
\bigl(\iota_r(\lambda\I_{\C(\mathcal{G})})\bigr)
=\Phi_r\bigl(A_{\scriptscriptstyle\mu_{r-1}(\mathcal{G})}\iota_0(\lambda)\bigr).
\end{align*}
Similarly for $a\in\C(\mathcal{G})$ we have
\begin{align*}
(A_{\scriptscriptstyle\mu_{r-1}(\mathcal{G})}\tens\id_{\C(\QAut(\mathcal{G}))})
&\Phi_r\bigl(\iota_1(a)\bigr)
=(A_{\scriptscriptstyle\mu_{r-1}(\mathcal{G})}\iota_1\tens\id_{\C(\QAut(\mathcal{G}))})\rho_{\scriptscriptstyle\mathcal{G}}(a)\\
&=(\iota_1 A_{\scriptscriptstyle\mathcal{G}}\tens\id_{\C(\QAut(\mathcal{G}))})\rho_{\scriptscriptstyle\mathcal{G}}(a)
+(\iota_2 A_{\scriptscriptstyle\mathcal{G}}\tens\id_{\C(\QAut(\mathcal{G}))})\rho_{\scriptscriptstyle\mathcal{G}}(a)\\
&=(\iota_1\tens\id_{\C(\QAut(\mathcal{G}))})\rho_{\scriptscriptstyle\mathcal{G}}(A_{\scriptscriptstyle\mathcal{G}}a)
+(\iota_2\tens\id_{\C(\QAut(\mathcal{G}))})\rho_{\scriptscriptstyle\mathcal{G}}(A_{\scriptscriptstyle\mathcal{G}}a)\\
&=\Phi_r\bigl(\iota_1(A_{\scriptscriptstyle\mathcal{G}}a)+\iota_2(A_{\scriptscriptstyle\mathcal{G}}a)\bigr)
=\Phi_r\bigl(A_{\scriptscriptstyle\mu_{r-1}(\mathcal{G})}\iota_1(a)\bigr),
\end{align*}
as well as
\begin{align*}
(A_{\scriptscriptstyle\mu_{r-1}(\mathcal{G})}\tens\id_{\C(\QAut(\mathcal{G}))})
&\Phi_r\bigl(\iota_l(a)\bigr)
=(A_{\scriptscriptstyle\mu_{r-1}(\mathcal{G})}\iota_l\tens\id_{\C(\QAut(\mathcal{G}))})\rho_{\scriptscriptstyle\mathcal{G}}(a)\\
&=(\iota_{l-1}A_{\scriptscriptstyle\mathcal{G}}\tens\id_{\C(\QAut(\mathcal{G}))})\rho_{\scriptscriptstyle\mathcal{G}}(a)
+(\iota_{l+1}A_{\scriptscriptstyle\mathcal{G}}\tens\id_{\C(\QAut(\mathcal{G}))})\rho_{\scriptscriptstyle\mathcal{G}}(a)\\
&=(\iota_{l-1}\tens\id_{\C(\QAut(\mathcal{G}))})\rho_{\scriptscriptstyle\mathcal{G}}(A_{\scriptscriptstyle\mathcal{G}}a)
+(\iota_{l+1}\tens\id_{\C(\QAut(\mathcal{G}))})\rho_{\scriptscriptstyle\mathcal{G}}(A_{\scriptscriptstyle\mathcal{G}}a)\\
&=\Phi_r\bigl(\iota_{l-1}(A_{\scriptscriptstyle\mathcal{G}}a)+\iota_{l+1}(A_{\scriptscriptstyle\mathcal{G}}a)\bigr)
=\Phi_r\bigl(A_{\scriptscriptstyle\mu_{r-1}(\mathcal{G})}\iota_l(a)\bigr)
\end{align*}
for $l=2,\dotsc,r-1$, and finally using the fact that $\eta_{\scriptscriptstyle\mathcal{G}}^*=\psi$ and that $\psi$ is $\rho_{\scriptscriptstyle\mathcal{G}}$-invariant we arrive at
\begin{align*}
(A_{\scriptscriptstyle\mu_{r-1}(\mathcal{G})}\:\tens\:&\id_{\C(\QAut(\mathcal{G}))})
\Phi_r\bigl(\iota_r(a)\bigr)
=(A_{\scriptscriptstyle\mu_{r-1}(\mathcal{G})}\iota_r\tens\id_{\C(\QAut(\mathcal{G}))})\rho_{\scriptscriptstyle\mathcal{G}}(a)\\
&=\delta_{\scriptscriptstyle\mathcal{G}}(\iota_0\eta_{\scriptscriptstyle\mathcal{G}}^*\tens\id_{\C(\QAut(\mathcal{G}))})\rho_{\scriptscriptstyle\mathcal{G}}(a)
+(\iota_{r-1}A_{\scriptscriptstyle\mathcal{G}}\tens\id_{\C(\QAut(\mathcal{G}))})\rho_{\scriptscriptstyle\mathcal{G}}(a)\\
&=\delta_{\scriptscriptstyle\mathcal{G}}(\iota_0\tens\id_{\C(\QAut(\mathcal{G}))})(\psi\tens\id_{\C(\QAut(\mathcal{G}))})\rho_{\scriptscriptstyle\mathcal{G}}(a)+(\iota_{r-1}\tens\id_{\C(\QAut(\mathcal{G}))})\rho_{\scriptscriptstyle\mathcal{G}}(A_{\scriptscriptstyle\mathcal{G}}a)\\
&=\delta_{\scriptscriptstyle\mathcal{G}}(\iota_0\tens\id_{\C(\QAut(\mathcal{G}))})\bigl(\psi(a)\tens\I_{\C(\QAut(\mathcal{G}))}\bigr)+(\iota_{r-1}\tens\id_{\C(\QAut(\mathcal{G}))})\rho_{\scriptscriptstyle\mathcal{G}}(A_{\scriptscriptstyle\mathcal{G}}a)\\
&=\delta_{\scriptscriptstyle\mathcal{G}}\bigl(\iota_0(\eta_{\scriptscriptstyle\mathcal{G}}^*a)\tens\I_{\C(\QAut(\mathcal{G}))}\bigr)+(\iota_{r-1}\tens\id_{\C(\QAut(\mathcal{G}))})\rho_{\scriptscriptstyle\mathcal{G}}(A_{\scriptscriptstyle\mathcal{G}}a)\\
&=\Phi_r\bigl(\delta_{\scriptscriptstyle\mathcal{G}}\iota_0(\eta_{\scriptscriptstyle\mathcal{G}}^*a)+\iota_{r-1}(a)\bigr)
=\Phi_r\bigl(A_{\scriptscriptstyle\mu_{r-1}(\mathcal{G})}\iota_r(a)\bigr).
\end{align*}

Finally, we note that if $\{e_1,\dotsc,e_{\dim{\C(\mathcal{G})}}\}$ is an orthonormal basis of $\C(\mathcal{G})$ then for any $k\in\{1,\dotsc,r\}$
\[
\Phi_r\bigl(\iota_k(e_j)\bigr)=(\iota_k\tens\id_{\C(\QAut(\mathcal{G}))})\rho_{\scriptscriptstyle\mathcal{G}}(e_j)=\sum_{i=1}^r\iota_k(e_i)\tens u_{i,j}
\]
showing that $\bigl\{(\omega\tens\id_{\C(\QAut(\mathcal{G}))})\Phi_r(x)\,\bigr|\bigl.\,x\in\C(\mu_{r-1}(\mathcal{G})),\:\omega\in\C(\mu_{r-1}(\mathcal{G}))^*\bigr\}$ generates \mbox{\cst-algebra} $\C(\QAut(\mathcal{G}))$.
\end{proof}

We can amplify actions of quantum groups on $\C(\mathcal{G})$ into those on $\C(\mathcal{G})\tens\B(\mathscr{H})$ for a Hilbert space $\mathscr{H}$. In fact, the action $\rho_{\scriptscriptstyle\mathcal{G}}$ of $\QAut(\mathcal{G})$ on $\mathcal{G}$ defines an action $\rho_{\scriptscriptstyle\mathcal{G},\mathscr{H}}$ of $\QAut(\mathcal{G})$ on $\C(\mathcal{G})\tens\B(\mathscr{H})$ by acting on $\B(\mathscr{H})$ trivially, that is $\rho_{\scriptscriptstyle\mathcal{G},\mathscr{H}}=(\id_{\C(\mathcal{G})}\tens\flip)(\rho_{\scriptscriptstyle\mathcal{G}}\tens\id_{\B(\mathscr{H})})$, where $\flip$ is the flip map. For a partition of unity $\mathfrak{P}\subset\C(\mathcal{G})\tens\B(\mathscr{H})$, we consider the action $\bigl(\rho_{\scriptscriptstyle\mathcal{G},\mathscr{H}}\bigr)_{\mathfrak{P}}$ of $\QAut(\mathcal{G})_{\mathfrak{P}}$ on $\C(\mathcal{G})\tens\B(\mathscr{H})$.

\begin{lemma}\label{partunit}
Let $\mathfrak{P}=\{P_1,\ldots, P_n\}$ be a partition of unity in $\C(\mathcal{G})\tens\B(\mathscr{H})$, and take $r\geq 2$. Then
\[
\begin{aligned}
    Q_0&=\iota_0\iota^*_0\tens\I_{\B(\mathscr{H})},\\
    Q_a&=\sum\limits_{k=1}^r(\iota_k\tens\I_{\B(\mathscr{H})})P_a(\iota_k^*\tens\I_{\B(\mathscr{H})}),\quad a=1,\dotsc,n.
\end{aligned}
\]
form a partition of unity $\mathfrak{P}_{\mu}^{r-1}$ in $\C(\mu_{r-1}(\mathcal{G}))\tens\B(\mathscr{H})$.
\end{lemma}
\begin{proof}
First, notice that $Q_0$ and $Q_a$ with $a=1,\ldots,n$ are projections. Indeed, we immediately notice that both $Q_0$ and $Q_a$ with $a=1,\ldots,n$ are idempotent and self-adjoint.Finally, we note that
\[
\begin{aligned}
\sum\limits_{Q\in\mathfrak{P}_\mu^{r-1}}\!\!\!\!Q&=Q_0+\sum\limits_{a=1}^n Q_a^k =\bigl(\iota_0\tens\I_{\B(\mathscr{H})}\bigr)\bigl(\iota_0^\ast\tens\I_{\B(\mathscr{H})}) +\sum\limits_{k=1}^r\sum\limits_{a=1}^n\bigl(\iota_k\tens\I_{\B(\mathscr{H})}\bigr) P_a(\iota_k^\ast\tens\I_{\B(\mathscr{H})})\\
&=\Bigl(\iota_0\iota_0^\ast +\sum\limits_{k=1}^r\iota_k\iota_k^\ast\Bigr)\tens\I_{\B(\mathscr{H})}=\I_{\C\bigl(\mu_{r-1}(G)\bigr)\tens\B(\mathscr{H})}.
\end{aligned}
   \]
\end{proof}
We refer to $\mathfrak{P}_{\mu}^{r-1}$ as the {\it induced partition of unity}. For $r=2$, we write $\mathfrak{P}_{\mu}\equiv\mathfrak{P}_{\mu}^1$. If $\mathscr{H}$ is one-dimensional, we deal with partitions of unity in the $\cst$-algebra $\C(\mathcal{G})$.
\begin{corollary}
\label{cor:subgroup}
Let $\mathcal{G}$ be a quantum graph and $\mathfrak{P} =\{P_1,\ldots,P_n\}\subset\C(\mathcal{G})\tens\B(\mathscr{H})$ be a partition of unity. Then for any $r\geq{2}$, the compact quantum group $\QAut(\mathcal{G})_{\mathfrak{P}}$ is embedded as a closed quantum subgroup of $\QAut(\mu_{r-1}(\mathcal{G}))_{\mathfrak{P}_{\mu}^{r-1}}$.
\end{corollary}
\begin{proof}
Let $\Phi_{r,\mathscr{H}}$ be the homomorphism constructed in Proposition~\ref{prop:Phir} extended onto $\C(\mu_{r-1}(\mathcal{G}))\tens\B(\mathscr{H})$, that is for $x\in\C(\mu_{r-1}(\mathcal{G}))\tens\B(\mathscr{H})$ we define
\[
\Phi_{r,\mathscr{H}}(x)=\bigl(\iota_0\iota_0^\ast\tens\id_{\B(\mathscr{H})}\bigr)(x)\tens\I_{\C(\QAut(\mathcal{G}))}+\sum\limits_{k=1}^r\bigl(\iota_k\tens\id_{\B(\mathscr{H})}\tens\id_{\C(\QAut(\mathcal{G}))}\bigr)\rho_{\mathcal{G},\mathscr{H}}\bigl(\iota_k^\ast\tens\id_{\B(\mathscr{H})}\bigr)(x).
\]
Further, we note that
\[
   \Phi_{r,\mathscr{H}}(Q_0)=\bigl(\iota_0\iota_0^\ast\tens\id_{\B(\mathscr{H})}\bigr)\bigl(\iota_0\iota_0^\ast\tens\I_{\B(\mathscr{H})}\bigr)\tens\I_{\C(\QAut(\mathcal{G}))}= Q_0\tens\I_{\C(\QAut(\mathcal{G}))},
\]
and
\[
\begin{aligned}
   \Phi_{r,\mathscr{H}}(Q_a)&=
   \sum\limits_{k,l=1}^r\bigl(\iota_k\tens\id_{\B(\mathscr{H})}\tens\id_{\C(\QAut(\mathcal{G}))}\bigr)\rho_{\mathcal{G},\mathscr{H}}\bigl(\iota_k^\ast\tens\id_{\B(\mathscr{H})}\bigr)\bigl(\iota_l\tens\I_{\B(\mathscr{H})}\bigr)P_a\bigl(\iota_l^\ast\tens\I_{\B(\mathscr{H})}\bigr)
   \\&=\sum\limits_{k=1}^r\bigl(\iota_k\tens\id_{\B(\mathscr{H})}\tens\id_{\C(\QAut(\mathcal{G}))}\bigr)\rho_{\mathcal{G},\mathscr{H}}(P_a) (\iota_k^\ast\tens\I_{\B(\mathscr{H})}\tens\I_{\C(\QAut(\mathcal{G}))}).
\end{aligned}
\]
If $\rho_{\scriptscriptstyle\mathcal{G},\mathscr{H}}(P_a)=P_a\tens\I_{\C(\QAut(\mathcal{G}))}$, then the above expression simplifies into $Q_a\tens\I_{\C(\QAut(\mathcal{G}))}$,
so that $\Phi_{r,\mathscr{H}}$ induces a homomorphism
\[
\Phi_{r,\mathscr{H}}^{\mathfrak{P}}\colon\C(\mu_{r-1}(\mathcal{G}))\tens\B(\mathscr{H})\longrightarrow\C(\mu_{r-1}(\mathcal{G}))\tens\B(\mathscr{H})\tens\C(\QAut(\mathcal{G})_{\mathfrak{P}}).
\]
By the universal property of $\QAut(\mu_{r-1}(\mathcal{G}))$ there exists a unique homomorphism\[\pi_r^{\mathfrak{P}}\colon\C(\QAut(\mu_{r-1}(\mathcal{G})))\longrightarrow\C(\QAut(\mathcal{G})_{\mathfrak{P}})\]such that $\Phi_{r,\mathscr{H}}^{\mathfrak{P}}=(\id\tens\pi_r^{\mathfrak{P}})\comp\rho_{\scriptscriptstyle\mu_{r-1}(\mathcal{G}),\mathscr{H}}$. Moreover, $\pi_r^{\mathfrak{P}}$ restricts itself to a well-defined homomorphism
$\pi_r^{\mathfrak{P},\mathfrak{P}_{\mu}^{r-1}}\colon\C(\QAut(\mu_{r-1}(\mathcal{G}))_{\mathfrak{P}_{\mu}^{r-1}})\to\C(\QAut(\mathcal{G})_{\mathfrak{P}})$ and this homomorphism is compatible with comultiplications (cf.~\cite[Proposition 4.7]{qFam}) and is surjective by the last statement of Proposition~\ref{prop:Phir}.
\end{proof}

\subsection{The role of twin vertices}\label{subsec:twin}\hspace*{\fill}

\begin{definition}%\label{def:twins0}
Let $G=(V,E)$ be a classical graph. For $v\in V$ define $N(v)=\bigl\{u\in{V}\,\bigr|\bigl.\,\{u,v\}\in{E}\bigr\}$, the neighbourhood of $v$. Two vertices $x,y\in V$ are said to be \emph{twins} if $N(x)=N(y)$.
\end{definition}

\begin{theorem}[{\cite[Proposition 4.1]{Alikhani19}}]\label{thm:main-class}
If $G=(V,E)$ is a twin-free finite graph, then every automorphism of $\mu(G)$ that fixes the point $\bullet$ is induced by an automorphism of $G$.
\end{theorem}

We will present here a slightly modified proof of this theorem which makes more explicit use of the adjacency matrix of $G$, allowing us to extend the reasoning to the case of quantum graphs in Theorem~\ref{auto_M}.

\begin{proof}[Proof of Theorem~\ref{thm:main-class}]
Let $\widetilde{\rho}$ be such an automorphism of $\mu(G)$ that fixes $\bullet$. It follows that the matrix of $\widetilde{\rho}$ as a linear map $\CC\oplus\C(V)\oplus\C(V)\to\CC\oplus\C(V)\oplus\C(V)$ must be of the form
\[
\widetilde{\rho}=\begin{bmatrix}
1 & 0 & 0\\
0 &\rho_{11} &\rho_{12}\\
0 &\rho_{21} &\rho_{22}
\end{bmatrix}.
\]
From the commutation with the adjacency matrix,
\[
A_{\scriptscriptstyle\mu(G)}=\begin{bmatrix}
0 & 0 & 1\\
0 & A_{\scriptscriptstyle G} & A_{\scriptscriptstyle G}\\
1& A_{\scriptscriptstyle G} &0
\end{bmatrix}
\]
we deduce that $\rho_{12}=\rho_{21}=0$. Denoting for $i=1,2$$,\rho_{i}=\rho_{ii}\in\Aut(G)$, treated as operators on $\C(V)$, the remaining nontrivial conditions are
\begin{equation}\label{eq:rel_Q}
\rho_{2}A_{\scriptscriptstyle G}=A_{\scriptscriptstyle G}\rho_{1}=\rho_{1}A_{\scriptscriptstyle G}=A_{\scriptscriptstyle G}\rho_{2}.
\end{equation}
Now, evaluating both sides of the equation $\rho_1 A_{\scriptscriptstyle G}=\rho_2 A_{\scriptscriptstyle G}$ on $\delta_v$, the characteristic function of the vertex $v\in V$, we end up with equation $\rho_1\chi_{N(v)}=\rho_2\chi_{N(v)}$, where $\chi_{C}$ is the indicator function of the set $C$. Since $\rho_1,\rho_2$ are automorphisms of the classical graph $G$, there exist permutations $\tau,\sigma$ of $V$ such that $\rho_1\chi_{N(v)}=\chi_{N(\tau(v))}$ and $\rho_2\chi_{N(v)}=\chi_{N(\sigma(v))}$. The lack of twin vertices implies that $\sigma=\tau$, so that $\rho_1=\rho_2$, i.e.~every automorphism of $\mu(G)$ that has $\bullet$ as a fixed point is induced by an automorphism of $G$.
\end{proof}

\begin{remark}\label{rem:detAzero}
Note that in the proof presented above the conclusion $\rho_1=\rho_2$ can be immediately reached from \eqref{eq:rel_Q} if we assume that $A_{\scriptscriptstyle G}$ is invertible. This is of course a stronger assumption than $G$ being twin-free.
\end{remark}

\begin{remark}
A graph $G=(V,E)$ has twin vertices if and only if its adjacency matrix $A_{\scriptscriptstyle G}$ has at least two same rows or, equivalently, columns. The existence of such two rows is equivalent to the existence of a permutation $\pi$ of the set $V$ such that $P_\pi A_{\scriptscriptstyle G}=A_{\scriptscriptstyle G}$, where $P_\pi$ is the corresponding permutation matrix. In contrast, the presence of the same columns is equivalent to the existence of $\sigma$ such that $A_{\scriptscriptstyle G} P_\sigma =A_{\scriptscriptstyle G}$.
\end{remark}

Motivated by the above observation, we introduce the following definition:

\begin{definition}\label{defin_twins1}
Let $\mathcal{G}=(\mathbb{V}_{\scriptscriptstyle\mathcal{G}},\psi_{\scriptscriptstyle\mathcal{G}}, A_{\scriptscriptstyle\mathcal{G}})$ be a quantum graph and $\HH$ be a compact quantum group. We say that $\mathcal{G}$ has \emph{$\HH$-twin vertices}, or simply \emph{$\HH$-twins}, if there exists a pair of different actions of $\HH$ on $\mathcal{G}$
\[
\rho_1,\rho_2\colon\C(\mathcal{G})\to\C(\mathcal{G})\tens\C(\HH)
\]
such that
\begin{equation}\label{pairrho}
\rho_{1}\circ A_{\scriptscriptstyle\mathcal{G}}=\rho_{2}\circ A_{\scriptscriptstyle\mathcal{G}}.
\end{equation}
We say that $\mathcal{G}$ has \emph{quantum twin vertices} if there exists a quantum group $\HH$ such that $\mathcal{G}$ has $\HH$-twins. If a classical group $H$ that meets the above conditions exists, then we say that $\mathcal{G}$ has \emph{classical twin vertices}.
\end{definition}

\begin{remark}
Suppose that $\mathcal{G}$ and $\HH$ are classical, i.e.~$\HH$ can be identified with a compact group $H$. Suppose that $\mathcal{G}$ has $H$-twins, i.e.~there exist $\rho_1,\rho_2\colon\C(\mathcal{G})\to\C(\mathcal{G})\tens\C(H)$ satisfying \eqref{pairrho}. The actions $\rho_1$ and $\rho_2$ can be viewed as homomorphisms $H\ni h\mapsto\rho_i(h)\in\Aut(\mathcal{G})$. Since $\rho_1\neq\rho_2$ there is $h\in H$ such that $\rho_1(h)\neq\rho_2(h)$ and
\[
\rho_{1}(h)\circ A_{\scriptscriptstyle\mathcal{G}}=\rho_{2}(h)\circ A_{\scriptscriptstyle\mathcal{G}}.
\]
Defining $\sigma =\rho_{2}(h)^{-1}\rho_{1}(h)\in\Aut(\mathcal{G})$ we have $P_\sigma A_{\scriptscriptstyle\mathcal{G}} = A_{\scriptscriptstyle\mathcal{G}}$. Since $\sigma\neq\id$ we can find two vertices $v_1\neq v_2$ and $\sigma(v_1) = v_2$ and it is easy to see that $v_1$ and $v_2$ are twins.

Conversely, if $v_1,v_2\in\mathbb{V}_\mathcal{G}$ are twins we can define an order two symmetry $\pi\in\Aut(\mathcal{G})$ such that $\pi(v) = v$ for all $v\in\mathbb{V}\setminus\{v_1,v_2\}$ and $\pi(v_1) = v_2$ (note that this is a different way of saying that $v_1$ and $v_2$ are twin vertices). Taking $\HH =\mathbb{Z}_2=\{e,h\}$ and putting $\rho_1(h)=\pi$ and $\rho_2(h) =\id$ we see that $\mathcal{G}$ has $\mathbb{Z}_2$-twins in the sense of Definition~\ref{defin_twins1}.
\end{remark}

Next, we generalize Theorem~\ref{thm:main-class} to the case of quantum graphs, with a possible amplification by $\B(\mathscr{H})$. In the next Theorem we use partitions of unity and the corresponding construction of $\mathfrak{P}_{\mu}$  as described in Lemma \ref{partunit} and the notation below. In what follows $\K(\mathscr{H})$ denoted the algebra of compact operators on the Hilbert space $\mathscr{H}$ and $\M(\cdot)$ is the multiplier functor. The reader will find more on these concepts and their use in introductory sections of \cite{Woronowicz95} and \cite{unbo}.

\begin{theorem}\label{auto_M}
Let $\mathcal{G}=(\mathbb{V}_{\scriptscriptstyle\mathcal{G}},\psi_{\scriptscriptstyle\mathcal{G}}, A_{\scriptscriptstyle\mathcal{G}})$ be a quantum graph that does not have quantum twin vertices.   Let $\mathfrak{P}$ be a partition of unity in $\C(\mathcal{G})\tens\B(\mathscr{H})$ with some (separable) Hilbert space $\mathscr{H}$. Then $\QAut(\mu(\mathcal{G}))_{\mathfrak{P}_{\mu}}\cong\QAut(\mathcal{G})_{\mathfrak{P}}$.
\end{theorem}
\begin{proof}
Let us first recall that $\QAut(\mathcal{G})_{\mathfrak{P}}$ can be viewed as a closed quantum subgroup of $\QAut(\mu(\mathcal{G}))_{\mathfrak{P}_{\mu}}$ (cf. Corollary~~\ref{cor:subgroup}). 
We shall prove that the two quantum groups are in fact identical. 

So let us consider a partition of unity $\mathfrak{P}\subset\C( \mathcal{G})\otimes\B(\mathscr{H})$. Note that  the partition  $\mathfrak{P}_{\mu}\subset\C(\mu(\mathcal{G}))\otimes\B(\mathscr{H})$ assigned to $\mathfrak{P}$  contains the projection $P_\bullet\otimes \I_{\B(\mathscr{H})}\in\C(\mu(\mathcal{G}))\otimes \B(\mathscr{H})$ where  $P_\bullet=\left[\begin{smallmatrix}1\\0\\0\end{smallmatrix}\right]\in\C(\mu(\mathcal{G}))$. Let $\rho_{\scriptscriptstyle\mu(\mathcal{G}),\mathscr{H}}$  be the   amplification of the canonical action of  $\QAut(\mu(\mathcal{G}))_{\mathfrak{P}_{\mu}}$ on $\C(\mu(\mathcal{G}))$. The  direct sum structure of $\C(\mu(\mathcal{G}))$ gives rise to the ``matrix structure'' of  $\rho_{{\scriptscriptstyle\mu(\mathcal{G}),\mathscr{H}}}=[\rho_{ij}]_{i,j=0,1,2}$ where  using notation $\C(\mathbb{X}_0)=\CC$, $\C(\mathbb{X}_1)=\C(\mathcal{G})$ and $\C(\mathbb{X}_2)=\C(\mathcal{G})$ we define $\rho_{ij}:\C(\mathbb{X}_j)\otimes\K(\mathscr{H})\to\M(\C(\mathbb{X}_i)\otimes \K(\mathscr{H})\otimes \C(\QAut(\mu(\mathcal{G}))_{\mathfrak{P}_{\mu}})$ by the formula 
\[
\rho_{ij} = \bigl(\iota_j^*\otimes\id_{\B(\mathscr{H})}\otimes \id_{\C(\QAut(\mu(\mathcal{G}))_{\mathfrak{P}_{\mu}})}\bigr)\circ\rho_{{\scriptscriptstyle\mu(\mathcal{G}),\mathscr{H}}}\circ\bigl(\iota_i\otimes\id_{\B(\mathscr{H})}\bigr).
\]
Remembering that
\[
\rho_{{\scriptscriptstyle\mu(\mathcal{G}),\mathscr{H}}}\circ\iota_0(\lambda) = \sqrt{1+2\delta_{\scriptscriptstyle\mathcal{G}}^2}\,\lambda \bigl(P_{\bullet}\otimes\I_{\B(\mathscr{H})}\otimes\I_{\C(\QAut(\mu(\mathcal{G}))_{\mathfrak{P}_{\mu}})}\bigr)
\]
for every $\lambda\in\mathbb{C}$ and using $\iota_{j}^*(P_{\bullet}) = 0$ for $j\in\{1,2\}$ we get $\rho_{j0} =0$ and and $\rho_{00}(\lambda) = \lambda \bigl(P_{\bullet}\otimes\I_{\B(\mathscr{H})}\otimes\I_{\C(\QAut(\mu(\mathcal{G}))_{\mathfrak{P}_{\mu}})}\bigr)$ for every $\lambda\in\CC$. 
Using the $*$-condition of the $*$-homomorphism $\rho_{\scriptscriptstyle\mu(\mathcal{G}),\mathscr{H}}$ we get $\rho_{01}=0=\rho_{02}$.

Having the above facts established we can now prove that $\rho_{12}=\rho_{21}=0$. Indeed, since the adjacency matrix  $A_{\scriptscriptstyle\mu(\mathcal{G})}$ is of the form 
\[
A_{\scriptscriptstyle\mu(\mathcal{G})}=\delta_{\scriptscriptstyle\mathcal{G}}\iota_2\eta_{\scriptscriptstyle\mathcal{G}}\iota_0^\ast +\delta_{\scriptscriptstyle\mathcal{G}}\iota_0\eta_{\scriptscriptstyle\mathcal{G}}
^\ast\iota_2^\ast +\iota_1 A_{\scriptscriptstyle\mathcal{G}}\iota_1^\ast +\iota_1 A_{\scriptscriptstyle\mathcal{G}}\iota_2^\ast +\iota_2 A_{\scriptscriptstyle\mathcal{G}}\iota_1^\ast.
\]
(cf.~Section \ref{sect:intro}) we get  
\[
A_{\scriptscriptstyle\mu(\mathcal{G})}\bigl(\iota_0(1)\bigr)
=\delta_{\scriptscriptstyle\mathcal{G}}\iota_2(\I_{\C(\mathcal{G})})
\]
and therefore
\[
\bigl(\iota_{1}^*\otimes \id_{\B(\mathscr{H})}\otimes\id_{\C(\mu(\mathcal{G})}\bigr)
\rho_{{\scriptscriptstyle\mu(\mathcal{G}),\mathscr{H}}}
\bigl( A_{\scriptscriptstyle\mu(\mathcal{G})}(\iota_0(1))\otimes\I_{\B(\mathscr{H})}\bigr)
=\delta_{\scriptscriptstyle\mathcal{G}}\rho_{12}\bigl(\I_{\C(\mathcal{G})}\otimes\I_{\B(\mathscr{H})}\bigr).
\]
A similar computation but with the amplified adjacency matrix $A_{\scriptscriptstyle\mu(\mathcal{G})}$ to the left of the action   $\rho_{\scriptscriptstyle\mu(\mathcal{G}),\mathscr{H}}$ shows that 
the left hand side of the above equality vanishes, and consequently $\rho_{12}(\I)=0$. This implies that $\rho_{12}=\rho_{21}=0$.

Using the vanishing of $\rho_{ij},\rho_{0j},\rho_{i0}$ for $i,j\in\{1,2\}$  proven so far and the equality
\[
\rho_{{\scriptscriptstyle\mu(\mathcal{G}),\mathscr{H}}}\bigl(A_{\scriptscriptstyle\mu(\mathcal{G})}\iota_2(x)\otimes\I_{\B(\mathscr{H})}\bigr)
=\bigl(A_{\scriptscriptstyle\mu(\mathcal{G})}\otimes\id_{\K(\mathscr{H})}\otimes\id_{\C(\mu(\mathcal{G}))}\bigr) \rho_{{\scriptscriptstyle\mu(\mathcal{G}),\mathscr{H}}}\bigl(\iota_2(x)\otimes\I_{\B(\mathscr{H})}\bigr)
\]
which holds for every $x\in\C(\mathcal{G})$ we easily conclude that   $\rho_{11}(A_{\scriptscriptstyle\mathcal{G}}\otimes\I_{\B(\mathscr{H})})=\rho_{22}(A_{\scriptscriptstyle\mathcal{G}}\otimes\I_{\B(\mathscr{H})})$. ``Deamplifying'' the later  equality an referring to the assumed  lack of quantum twins in $\mathcal{G}$ we conclude that  $\rho_{11} = \rho_{22}$. 

Summarizing we have so far proved that the canonical action of $\QAut(\mu(\mathcal{G}))_{\mathfrak{P}_{\mu}}$ on $\C(\mu(\mathcal{G}))$  is induced (in the sense described in Section \ref{sec:symm}) by the action of $\QAut(\mathcal{G})$. Moreover using the fact that $\rho_{\scriptscriptstyle\mu(\mathcal{G})}$ preserves $\mathfrak{P}_{\mu}$ we see that in fact it is induced by the action of $\QAut(\mathcal{G})_{\mathfrak{P}}$ and therefore  we conclude that   $\QAut(\mathcal{G})_{\mathfrak{P}}$ and  $\QAut(\mu\mathcal{G})_{\mathfrak{P}_\mu}$ are indeed isomorphic.  
\end{proof}

\subsection{Quantum twin vertices in small classical graphs}\label{subsec:small}\hspace*{\fill}

It is unclear whether the notion of quantum twin vertices is broader than the usual concept of twin vertices when applied to classical graphs. In this section, we consider this question for graphs with few vertices. In view of Remark~\ref{rem:detAzero} we will be mainly interested in graphs with no twins, but whose adjacency matrix is not invertible.

Considering graphs up to the classical isomorphism, we begin with the following observation.

\begin{proposition}
There is no graph $G$ with $|V(G)|\leq 5$ such that $G$ has quantum symmetries, has no twins, and $\det A_{\scriptscriptstyle G}=0$.
\end{proposition}

\begin{proof}
The statement is clear for $|V(G)|\leq 3$. For $|V(G)|=4$, only three graphs possess quantum symmetries, cf.~\cite{Weber22}. Their adjacency matrices are
\[
\begin{bmatrix}
0&0&1&1\\
0&0&1&1\\
1&1&0&0\\
1&1&0&0
\end{bmatrix},
\begin{bmatrix}
0&0&1&1\\
0&0&1&1\\
1&1&0&1\\
1&1&1&0
\end{bmatrix},
\begin{bmatrix}
0&1&1&1\\
1&0&1&1\\
1&1&0&1\\
1&1&1&0
\end{bmatrix}
\]
respectively. The first two have twins, as the corresponding adjacency matrices have two identical rows. However, the last one is twin-free, but the adjacency matrix has non-zero determinant. 

For $|V(G)|=5$, only ten graphs possess quantum symmetries (again see \cite{Weber22}). One of them is the complete graph $K_5$, and the adjacency matrices of the remaining ones are
\begin{align*}
&A_{\scriptscriptstyle G_1^5}=\begin{bmatrix}
0&0&0&0&1\\
0&0&0&0&1\\
0&0&0&0&1\\
0&0&0&0&1\\
1&1&1&1&0\\
\end{bmatrix},\quad
A_{\scriptscriptstyle G_2^5}=\begin{bmatrix}
0&0&0&0&1\\
0&0&0&0&1\\
0&0&0&1&1\\
0&0&1&0&1\\
1&1&1&1&0\\
\end{bmatrix},\quad
A_{\scriptscriptstyle G_3^5}=\begin{bmatrix}
0&0&0&1&1\\
0&0&0&1&1\\
0&0&0&1&1\\
1&1&1&0&0\\
1&1&1&0&0\\
\end{bmatrix},\\
&A_{\scriptscriptstyle G_4^5}=\begin{bmatrix}
0&0&0&1&1\\
0&0&0&1&1\\
0&0&0&1&1\\
1&1&1&0&1\\
1&1&1&1&0\\
\end{bmatrix},\quad
A_{\scriptscriptstyle G_5^5}=\begin{bmatrix}
0&0&0&1&1\\
0&0&1&0&1\\
0&1&0&0&1\\
1&0&0&0&1\\
1&1&1&1&0
\end{bmatrix},\quad
A_{\scriptscriptstyle G_6^5}=\begin{bmatrix}
0&0&0&1&1\\
0&0&1&1&1\\
0&1&0&1&1\\
1&1&1&0&0\\
1&1&1&0&0
\end{bmatrix},\\
&A_{\scriptscriptstyle G^5_7}=\begin{bmatrix}
0&0&0&1&1\\
0&0&1&1&1\\
0&1&0&1&1\\
1&1&1&0&1\\
1&1&1&1&0
\end{bmatrix},\quad
A_{\scriptscriptstyle G^5_8}=\begin{bmatrix}
0&0&1&1&1\\
0&0&1&1&1\\
1&1&0&0&1\\
1&1&0&0&1\\
1&1&1&1&0
\end{bmatrix},\quad
A_{\scriptscriptstyle G^5_9}=\begin{bmatrix}
0&0&1&1&1\\
0&0&1&1&1\\
1&1&0&1&1\\
1&1&1&0&1\\
1&1&1&1&0
\end{bmatrix}.
\end{align*}
Among them, only $K_5$, $G_5^5$ and $G_7^5$ have no twins. However, we have $\det A_{\scriptscriptstyle K_5}=4$, $\det A_{\scriptscriptstyle G_5^5}=-4$ and $\det A_{\scriptscriptstyle G_7^5}=-2$.
\end{proof}

For graphs with six vertices, we have the following proposition:

\begin{proposition}\label{prop:twin6}
For graphs with $|V(G)|\leq 6$, the only one (up to classical isomorphism) that has both quantum symmetries, no twins, and such that $\det A_{\scriptscriptstyle G}=0$ is described by the following adjacency matrix
\[
A_{\scriptscriptstyle G_6}=\begin{bmatrix}
0&0&0&0&1&1\\
0&0&1&1&0&1\\
0&1&0&1&0&1\\
0&1&1&0&1&1\\
1&0&0&1&0&0\\
1&1&1&1&0&0
\end{bmatrix}.
\]
\end{proposition}

We present the proof of Proposition~\ref{prop:twin6} in the Appendix. Here we notice that this graph has the property that for every pair of vertices sharing the same number of edges, there is an edge between them. For example, vertices 2 and 3 have three edges, and there is an edge connecting these two vertices. This property ensures that there are no classical twins.

For establishing the next result, we need the following lemma.
\begin{lemma}[{Fulton criterion; \cite{Fulton}, see also \cite[Lemma 2.6]{Weber22}}] For a finite graph $G$, if $\bigl(A_{\scriptscriptstyle G}^l\bigr)_{ii}=\bigl(A_{\scriptscriptstyle G}^{l}\bigr)_{jj}$ for some $l\in\mathbb{N}$, then $u_{i,j}$ -- the $(i,j)$-entry of the generating matrix of $\C(\QAut(G))$ -- vanishes.
\end{lemma}

\begin{proposition}
The graph $G_6$ from Proposition~\ref{prop:twin6} does not possess quantum twins.
\end{proposition}

\begin{proof}
First, by the Fulton criterion and the fact that for all $1\leq i\leq 6$ we have $\sum\limits_{k=1}^{6} u_{ik}=\sum\limits_{k=1}^6 u_{ki}=1$, we know that the generating matrix $u$ of $\C(\QAut(G_6))$ is of the form
\[
u=\begin{bmatrix}
u_{11} & 0 & 0 &0 & 1-u_{11} & 0\\
0 & u_{22} & 1-u_{22} & 0 &0 &0\\
0 & 1-u_{22} & u_{22} & 0 &0&0\\
0&0&0& u_{33} & 0 & 1-u_{33}\\
1-u_{11} & 0 &0 &0& u_{11} & 0\\
0&0&0&1-u_{33}&0&u_{33}
\end{bmatrix}.
\]
Suppose there exists $\rho_1,\rho_2$, a pair of different actions of a certain quantum group $\mathbb{H}$ on $G_6$ such that $\rho_1\circ A_{\scriptscriptstyle G_6}=\rho_2\circ A_{\scriptscriptstyle G_6}$. Let $P$ and $Q$ be matrix representations of $\rho_1$ and $\rho_2$ on the standard basis of $L^2(G)$, respectively. They can be parameterized as
\[
P=\begin{bmatrix}
p_1 & 0 & 0 &0 & 1-p_1 & 0\\
0 & p_2 & 1-p_2 & 0 &0 &0\\
0 & 1-p_2 & p_2 & 0 &0&0\\
0&0&0& p_3 & 0 & 1-p_3\\
1-p_1 & 0 &0 &0& p_1 & 0\\
0&0&0&1-p_3&0&p_3
\end{bmatrix}
\]
and
\[
Q=\begin{bmatrix}
q_1 & 0 & 0 &0 & 1-q_1 & 0\\
0 & q_2 & 1-q_2 & 0 &0 &0\\
0 & 1-q_2 & q_2 & 0 &0&0\\
0&0&0& q_3 & 0 & 1-q_3\\
1-q_1 & 0 &0 &0& q_1 & 0\\
0&0&0&1-q_3&0&q_3
\end{bmatrix},
\]
respectively, where $p_1,p_2,p_3$ and $q_1,q_2,q_3$ are orthogonal projections in $\C(\mathbb{H})$. Since they must commute with the adjacency matrix, we get $p_3=p_1$ and $q_3=q_1$. With this condition, the relation $\rho_1\circ A_{\scriptscriptstyle G_6}=\rho_2\circ A_{\scriptscriptstyle G_6}$ leads to $p_1=q_1$ and $p_2=q_2$, so $\rho_1$ is identical to $\rho_2$, a contradiction.
\end{proof}

\begin{corollary}
The notions of quantum and classical twin vertices coincide for graphs $G$ with $|V(G)|\leq 6$.
\end{corollary}

We now consider a subclass of classical graphs, the vertex-transitive graphs. We recall that a graph $G$ is called vertex-transitive if $\Aut(G)$ acts transitively on the set of its vertices. It is known\cite{Banica07,  Chassaniol, Chassaniol19, Schanz24, Schmidt18} exactly which of these graphs with $|V|\leq 13$ possess nonclassical quantum symmetries. Using these results, we have the following:

\begin{proposition}
For connected vertex-transitive graphs with at most $13$ vertices, only $K_2\Box K_5$ and $C_4\Box C_3$ possess non-classical quantum symmetries, have no twin vertices and vanishing determinant of the adjacency matrix.
\end{proposition}

\begin{proof}
We analyze the properties of such graphs with non-classical quantum symmetries, listed in \cite[Tables 1-2]{Schanz24}. We first observe that among them, the only ones with the vanishing determinant of the adjacency matrix are $K_2\Box K_5$, $C_{10}(4)$, $C_2\Box C_4$ and $C_{12}(5)$. However, $C_{10}(4)$ possesses five pairs of twin vertices, while $C_{12}(5)$ has six such pairs. For the remaining two graphs, we have
\[
A_{\scriptscriptstyle K_2\Box K_5}=A_{\scriptscriptstyle K_2}\tens\I_5 +\I_2\tens A_{\scriptscriptstyle K_5}=\begin{bmatrix}
0 & 1 & 1 & 1 & 1 & 1 & 0 & 0 & 0 & 0\\
1 & 0 & 1 & 1 & 1 & 0 & 1 & 0 & 0 & 0\\
1 & 1 & 0 & 1 & 1 & 0 & 0 & 1 & 0 & 0\\
1 & 1 & 1 & 0 & 1 & 0 & 0 & 0 & 1 & 0\\
1 & 1 & 1 & 1 & 0 & 0 & 0 & 0 & 0 & 1\\
1 & 0 & 0 & 0 & 0 & 0 & 1 & 1 & 1 & 1\\
0 & 1 & 0 & 0 & 0 & 1 & 0 & 1 & 1 & 1\\
0 & 0 & 1 & 0 & 0 & 1 & 1 & 0 & 1 & 1\\
0 & 0 & 0 & 1 & 0 & 1 & 1 & 1 & 0 & 1\\
0 & 0 & 0 & 0 & 1 & 1 & 1 & 1 & 1 & 0
\end{bmatrix}
\]
and
\[
A_{\scriptscriptstyle C_4\Box C_3}=A_{\scriptscriptstyle C_4}\tens\I_3 +\I_4\tens A_{\scriptscriptstyle C_3}
=\left[
\begin{array}{cccccccccccc}
 0 & 1 & 1 & 1 & 0 & 0 & 0 & 0 & 0 & 1 & 0 & 0\\
 1 & 0 & 1 & 0 & 1 & 0 & 0 & 0 & 0 & 0 & 1 & 0\\
 1 & 1 & 0 & 0 & 0 & 1 & 0 & 0 & 0 & 0 & 0 & 1\\
 1 & 0 & 0 & 0 & 1 & 1 & 1 & 0 & 0 & 0 & 0 & 0\\
 0 & 1 & 0 & 1 & 0 & 1 & 0 & 1 & 0 & 0 & 0 & 0\\
 0 & 0 & 1 & 1 & 1 & 0 & 0 & 0 & 1 & 0 & 0 & 0\\
 0 & 0 & 0 & 1 & 0 & 0 & 0 & 1 & 1 & 1 & 0 & 0\\
 0 & 0 & 0 & 0 & 1 & 0 & 1 & 0 & 1 & 0 & 1 & 0\\
 0 & 0 & 0 & 0 & 0 & 1 & 1 & 1 & 0 & 0 & 0 & 1\\
 1 & 0 & 0 & 0 & 0 & 0 & 1 & 0 & 0 & 0 & 1 & 1\\
 0 & 1 & 0 & 0 & 0 & 0 & 0 & 1 & 0 & 1 & 0 & 1\\
 0 & 0 & 1 & 0 & 0 & 0 & 0 & 0 & 1 & 1 & 1 & 0\\
\end{array}
\right].
\]
By a direct inspection, we verify that the above matrices do not have two identical rows so there are no classical twin vertices in these graphs.
\end{proof}

We note that the Fulton criterion does not yield any information about the graphs $K_2\Box K_5$ (the so-called {\it prism} having a basis in $K_5$, $\mathrm{Pr}(K_5)$ \cite[Definition 4.2]{Banica07}) and $C_4\Box C_3$ as they are vertex-transitive \cite[Remark 2.7]{Weber22}. The other potential candidate is the Tesseract graph with $16$ vertices. Its adjacency matrix is singular, while it does not have classical twins. The classical symmetry group is $H_4$, while the quantum symmetry group is $O_4^{-1}$, see \cite{BanBichColl07}.

\begin{question}
Do the above-mentioned graphs possess quantum twin vertices?
\end{question}

\section{The distinguishing number}\label{sec:dist}

A \emph{labeling} of a classical graph $G=(V,E)$ is a mapping $V\to\{1,\dotsc,c\}$, while a \emph{coloring} is a labeling with the property that adjacent vertices are labeled with different integers (colors). The \emph{chromatic number} of $G$ defined as the smallest $c$ such that there exists a coloring of $G$ with $c$ colors is an invariant that measures how ``well connected'' the graph $G$ is. A different invariant called \emph{distinguishing number} measures this property with respect to the size of the automorphism group of $G$. We recall here this classical definition:

\begin{definition}[{\cite{AlbertsonCollins96}}]\label{def:dist}
Let $G=(V,E)$ be a finite graph. A $c$-distinguishing labeling of $G$ is a function $\phi\colon V\to\{1,\dotsc,c\}$ such that for every $\sigma\in\Aut(G)\setminus\{\id\}$, there exists $v\in V$ such that $\phi(v)\neq\phi(\sigma(v))$. The distinguishing number $D(G)$ of $G$ is then the smallest $c$ such that $G$ has a labeling that is $c$-distinguishing.
\end{definition}

The behavior of the distinguishing number and variations of it under the Mycielski transformation was widely studied in recent years; see, e.g.~\cite{Alikhani19, Kennedy24}. Here we introduce a version of Definition~\ref{def:dist} for quantum graphs. To that end, we recall that a quantum graph $\mathcal{G}$ defines an operator system $S_{\scriptscriptstyle\mathcal{G}}\subset\B(L^2(\mathcal{G}))$ through $S_{\scriptscriptstyle\mathcal{G}}=\bigl\{m(A\tens{X})m^*\,\bigr|\bigl.\,X\in\B(L^2(\mathcal{G}))\bigr\}$ (see, e.g.~\cite[Section 2.4]{Ganesan}).

\begin{definition}
Let $\mathcal{G}$ be an irreflexive quantum graph, $S_{\scriptscriptstyle\mathcal{G}}$ the operator system associated with $\mathcal{G}$ and $\mathscr{H}$ be a Hilbert space. A partition of unity $\{P_1,\dotsc,P_c\}$ in $\C(\mathcal{G})\tens\B(\mathscr{H})$ is called a \emph{quantum $c$-labeling of $\mathcal{G}$}. We say that a quantum $c$-labeling of $\mathcal{G}$ is a \emph{quantum $c$-coloring of $\mathcal{G}$} if
\[
\forall\:a\in\{1,\dotsc,c\}\quad\forall\:X\in S_{\scriptscriptstyle\mathcal{G}}\quad P_a(X\tens\I_\mathcal{\mathscr{H}})P_a=0.
\]
\end{definition}

\begin{definition}
Let $\mathfrak{P}=\{P_1,\dotsc,P_c\}\subset\C(\mathcal{G})\tens\B(\mathscr{H})$ be a quantum $c$-labeling of $\mathcal{G}$. We say that $\mathfrak{P}$ is a \emph{quantum $c$-distinguishing labeling} of $\mathcal{G}$ if the quantum group $\QAut(\mathcal{G})_{\mathfrak{P}}$ is trivial.
\end{definition}
\begin{remark}
Note that if $\mathcal{G}$ is a classical graph (i.e.~the \cst-algebra $\C(\mathcal{G})$ is commutative) and $\mathfrak{P}$ is a classical $c$-labeling (i.e.~$\mathscr{H}$ is one dimensional) then it corresponds to labeling of vertices $f\colon{V_\mathcal{G}}\to\{1,\dotsc,c\}$ uniquely defined by the condition $P_i(v)=1\Leftrightarrow{f(v)}=i$. In this case $\QAut(\mathcal{G})_{\mathfrak{P}}$ can be viewed as the quantum group of automorphisms preserving $f$. In particular, the abelianization of $\C(\QAut(\mathcal{G})_{\mathfrak{P}})$ corresponds to the subgroup of $\Aut(\mathcal{G})$ that preserves $f$.
\end{remark}

\begin{definition}
Let $\mathcal{G}$ be a quantum graph. We define the quantum \emph{distinguishing number} $D(\mathcal{G})$ of $\mathcal{G}$ as the smallest $c$ such that $\mathcal{G}$ admits a quantum distinguishing $c$-labeling.
\end{definition}

\begin{theorem}
Let $\mathcal{G}$ be a quantum graph with $\dim\C(\mathcal{G})\geq 2$ which does not possess quantum twin vertices. Then $D(\mu(\mathcal{G}))\leq D(\mathcal{G})+1$.
\end{theorem}
\begin{proof}
Let $\mathfrak{P}=\{P_1,\dotsc,P_{D(\mathcal{G})}\}\subset\C(\mathcal{G})\tens\B(\mathscr{H})$ be a quantum $D(\mathcal{G})$-distinguishing labeling of $\mathcal{G}$, so that $\C(\QAut(\mathcal{G})_{\mathfrak{P}})=\CC$, and consider the induced partition of unity $\mathfrak{P}_{\mu}$ on $\C(\mu(\mathcal{G}))\tens\B(\mathscr{H})$, that is $\mathfrak{P}_{\mu}=\{Q_0,Q_1,\ldots, Q_{D(\mathcal{G})}\}$ with
\[
\begin{aligned}
    Q_0&=\iota_0\iota^*_0\tens\I_{\B(\mathscr{H})},\\
    Q_a&=(\iota_1\tens\I_{\B(\mathscr{H})})P_a(\iota_1^*\tens\I_{\B(\mathscr{H})}) + (\iota_2\tens\I_{\B(\mathscr{H})})P_a(\iota_2^*\tens\I_{\B(\mathscr{H})}),\quad a=1,\dotsc,D(\mathcal{G}).
\end{aligned}
\]
By Theorem~\ref{auto_M}, $\mathfrak{P}_{\mu}$ is a quantum $(D(\mathcal{G})+1)$-labeling of $\mu(\mathcal{G})$.
\end{proof}

\section*{Acknowledgments}
The work of AB was partially funded by the Deutsche Forschungsgemeinschaft (DFG, German Research Foundation) under Germany’s Excellence Strategy - EXC-2111 – 390814868. This research was partially supported by the University of Warsaw Thematic Research Program ``Quantum Symmetries". AB is supported by the Alexander von Humboldt Foundation. PS was supported by the NCN (National Science Center, Poland) grant no.~2022/47/B/ST1/00582. ICh is supported by the NCN (National Science Center, Poland) grant SONATA-17 no.~
UMO-2021/43/D/ST1/01446. Furthermore, the authors wish to thank Adam Skalski for the fruitful discussion on the subject of quantum symmetries of quantum graphs.

\section{Appendix: Proof of Proposition~\ref{prop:twin6}}

Here we present the proof of Proposition~\ref{prop:twin6} explicitly listing all $55$ graphs with six vertices that have quantum symmetries (\cite{Weber22}), and for each of them verifying if it possesses twins. In case the graph is twin-free, we provide instead the value of the determinant of the corresponding adjacency matrix:
\allowdisplaybreaks
\begin{align*}
A_{\scriptscriptstyle G_{1}^6}&=\begin{bmatrix}
0&0&0&0&0&1\\
0&0&0&0&0&1\\
0&0&0&0&0&1\\
0&0&0&0&0&1\\
0&0&0&0&0&1\\
1&1&1&1&1&0
\end{bmatrix}\quad\text{has twins},
&A_{\scriptscriptstyle G_{2}^6}&=\begin{bmatrix}
0&0&0&0&0&1\\
0&0&0&0&0&1\\
0&0&0&0&0&1\\
0&0&0&0&1&1\\
0&0&0&1&0&1\\
1&1&1&1&1&0
\end{bmatrix}\quad\text{has twins},
\\
A_{\scriptscriptstyle G_{3}^6}&=\begin{bmatrix}
0&0&0&0&0&1\\
0&0&0&0&0&1\\
0&0&0&0&1&0\\
0&0&0&0&1&0\\
0&0&1&1&0&1\\
1&1&0&0&1&0
\end{bmatrix}\quad\text{has twins},
&A_{\scriptscriptstyle G_{4}^6}&=\begin{bmatrix}
0&0&0&0&0&1\\
0&0&0&0&0&1\\
0&0&0&0&1&1\\
0&0&0&0&1&1\\
0&0&1&1&0&0\\
1&1&1&1&0&0
\end{bmatrix}\quad\text{has twins},
\\
A_{\scriptscriptstyle G_{5}^6}&=\begin{bmatrix}
0&0&0&0&0&1\\
0&0&0&0&0&1\\
0&0&0&0&1&1\\
0&0&0&0&1&1\\
0&0&1&1&0&1\\
1&1&1&1&1&0
\end{bmatrix}\quad\text{has twins},
&A_{\scriptscriptstyle G_{6}^6}&=\begin{bmatrix}
0&0&0&0&0&1\\
0&0&0&0&0&1\\
0&0&0&1&1&0\\
0&0&1&0&1&0\\
0&0&1&1&0&1\\
1&1&0&0&1&0
\end{bmatrix}\quad\text{has twins},
\\
A_{\scriptscriptstyle G_{7}^6}&=\begin{bmatrix}
0&0&0&0&0&1\\
0&0&0&0&0&1\\
0&0&0&1&1&0\\
0&0&1&0&1&1\\
0&0&1&1&0&1\\
1&1&0&1&1&0
\end{bmatrix}\quad\text{has twins},
&A_{\scriptscriptstyle G_{8}^6}&=\begin{bmatrix}
0&0&0&0&0&1\\
0&0&0&0&0&1\\
0&0&0&1&1&1\\
0&0&1&0&1&1\\
0&0&1&1&0&0\\
1&1&1&1&1&0
\end{bmatrix}\quad\text{has twins},
\\
A_{\scriptscriptstyle G_{9}^6}&=\begin{bmatrix}
0&0&0&0&0&1\\
0&0&0&0&1&0\\
0&0&0&0&1&1\\
0&0&0&0&1&1\\
0&1&1&1&0&0\\
1&0&1&1&0&0
\end{bmatrix}\quad\text{has twins},
&A_{\scriptscriptstyle G_{10}^6}&=\begin{bmatrix}
0&0&0&0&0&1\\
0&0&0&0&1&0\\
0&0&0&0&1&1\\
0&0&0&0&1&1\\
0&0&1&1&0&1\\
0&0&1&1&1&0
\end{bmatrix}\quad\text{has twins},
\\
A_{\scriptscriptstyle G_{11}^6}&=\begin{bmatrix}
0&0&0&0&0&1\\
0&0&0&0&1&0\\
0&0&0&1&1&1\\
0&0&1&0&1&1\\
0&1&1&1&0&0\\
1&0&1&1&0&0
\end{bmatrix},\quad\det=-1,
&A_{\scriptscriptstyle G_{12}^6}&=\begin{bmatrix}
0&0&0&0&0&1\\
0&0&0&0&1&0\\
0&0&0&1&1&1\\
0&0&1&0&1&1\\
0&1&1&1&0&1\\
1&0&1&1&1&0
\end{bmatrix},\quad\det=-1,
\\
A_{\scriptscriptstyle G_{13}^6}&=\begin{bmatrix}
0&0&0&0&0&1\\
0&0&0&0&1&1\\
0&0&0&1&0&1\\
0&0&1&0&0&1\\
0&1&0&0&0&1\\
1&1&1&1&1&0
\end{bmatrix},\quad\det=-1,
&A_{\scriptscriptstyle G_{14}^6}&=\begin{bmatrix}
0&0&0&0&0&1\\
0&0&0&1&1&0\\
0&0&0&1&1&0\\
0&1&1&0&0&1\\
0&1&1&0&0&1\\
1&0&0&1&1&0
\end{bmatrix}\quad\text{has twins},
\\
A_{\scriptscriptstyle G_{15}^6}&=\begin{bmatrix}
0&0&0&0&0&1\\
0&0&0&1&1&0\\
0&0&0&1&1&0\\
0&1&1&0&1&1\\
0&1&1&1&0&1\\
1&0&0&1&1&0
\end{bmatrix}\quad\text{has twins},
&A_{\scriptscriptstyle G_{16}^6}&=\begin{bmatrix}
0&0&0&0&0&1\\
0&0&0&1&1&1\\
0&0&0&1&1&1\\
0&1&1&0&0&1\\
0&1&1&0&0&1\\
1&1&1&1&1&0
\end{bmatrix}\quad\text{has twins},
\\
A_{\scriptscriptstyle G_{17}^6}&=\begin{bmatrix}
0&0&0&0&0&1\\
0&0&0&1&1&1\\
0&0&0&1&1&1\\
0&1&1&0&1&0\\
0&1&1&1&0&0\\
1&1&1&0&0&0
\end{bmatrix}\quad\text{has twins},
&A_{\scriptscriptstyle G_{18}^6}&=\begin{bmatrix}
0&0&0&0&0&1\\
0&0&0&1&1&1\\
0&0&0&1&1&1\\
0&1&1&0&1&1\\
0&1&1&1&0&1\\
1&1&1&1&1&0
\end{bmatrix}\quad\text{has twins},
\\
A_{\scriptscriptstyle G_{19}^6}&=\begin{bmatrix}
0&0&0&0&0&1\\
0&0&1&1&1&0\\
0&1&0&1&1&0\\
0&1&1&0&1&1\\
0&1&1&1&0&1\\
1&0&0&1&1&0
\end{bmatrix},\quad\det=3,
&
A_{\scriptscriptstyle G_{20}^6}&=\begin{bmatrix}
0&0&0&0&0&1\\
0&0&1&1&1&1\\
0&1&0&1&1&1\\
0&1&1&0&1&1\\
0&1&1&1&0&1\\
1&1&1&1&1&0
\end{bmatrix},\quad\det=3,
\\
A_{\scriptscriptstyle G_{21}^6}&=\begin{bmatrix}
0&0&0&0&1&1\\
0&0&0&0&1&1\\
0&0&0&0&1&1\\
0&0&0&0&1&1\\
1&1&1&1&0&0\\
1&1&1&1&0&0
\end{bmatrix}\quad\text{has twins},
&A_{\scriptscriptstyle G_{22}^6}&=\begin{bmatrix}
0&0&0&0&1&1\\
0&0&0&0&1&1\\
0&0&0&0&1&1\\
0&0&0&0&1&1\\
1&1&1&1&0&1\\
1&1&1&1&1&0
\end{bmatrix}\quad\text{has twins},
\\
A_{\scriptscriptstyle G_{23}^6}&=\begin{bmatrix}
0&0&0&0&1&1\\
0&0&0&0&1&1\\
0&0&0&1&0&1\\
0&0&1&0&0&1\\
1&1&0&0&0&0\\
1&1&1&1&0&0
\end{bmatrix}\quad\text{has twins},
&A_{\scriptscriptstyle G_{24}^6}&=\begin{bmatrix}
0&0&0&0&1&1\\
0&0&0&0&1&1\\
0&0&0&1&0&1\\
0&0&1&0&0&1\\
1&1&0&0&0&1\\
1&1&1&1&1&0
\end{bmatrix}\quad\text{has twins},
\\
A_{\scriptscriptstyle G_{25}^6}&=\begin{bmatrix}
0&0&0&0&1&1\\
0&0&0&0&1&1\\
0&0&0&1&0&1\\
0&0&1&0&0&1\\
1&1&0&0&0&0\\
1&1&1&1&0&0
\end{bmatrix}\quad\text{has twins},
&A_{\scriptscriptstyle G_{26}^6}&=\begin{bmatrix}
0&0&0&0&1&1\\
0&0&0&0&1&1\\
0&0&0&1&0&1\\
0&0&1&0&1&0\\
1&1&0&1&0&1\\
1&1&1&0&1&0
\end{bmatrix}\quad\text{has twins},
\\
A_{\scriptscriptstyle G_{27}^6}&=\begin{bmatrix}
0&0&0&0&1&1\\
0&0&0&0&1&1\\
0&0&0&1&1&1\\
0&0&1&0&1&1\\
1&1&1&1&0&0\\
1&1&1&1&0&0
\end{bmatrix}\quad\text{has twins},
&A_{\scriptscriptstyle G_{28}^6}&=\begin{bmatrix}
0&0&0&0&1&1\\
0&0&0&0&1&1\\
0&0&0&1&1&1\\
0&0&1&0&1&1\\
1&1&1&1&0&1\\
1&1&1&1&1&0
\end{bmatrix}\quad\text{has twins},
\\
A_{\scriptscriptstyle G_{29}^6}&=\begin{bmatrix}
0&0&0&0&1&1\\
0&0&0&1&1&1\\
0&0&0&1&1&1\\
0&1&1&0&0&0\\
1&1&1&0&0&0\\
1&1&1&0&0&0
\end{bmatrix}\quad\text{has twins},
&A_{\scriptscriptstyle G_{30}^6}&=\begin{bmatrix}
0&0&0&0&1&1\\
0&0&0&1&1&1\\
0&0&0&1&1&1\\
0&1&1&0&0&0\\
1&1&1&0&0&1\\
1&1&1&0&1&0
\end{bmatrix}\quad\text{has twins},
\\
A_{\scriptscriptstyle G_{31}^6}&=\begin{bmatrix}
0&0&0&0&1&1\\
0&0&0&1&1&1\\
0&0&0&1&1&1\\
0&1&1&0&1&1\\
1&1&1&1&0&0\\
1&1&1&1&0&0
\end{bmatrix}\quad\text{has twins},
&A_{\scriptscriptstyle G_{32}^6}&=\begin{bmatrix}
0&0&0&0&1&1\\
0&0&0&1&1&1\\
0&0&0&1&1&1\\
0&1&1&0&1&1\\
1&1&1&1&0&1\\
1&1&1&1&1&0
\end{bmatrix}\quad\text{has twins},
\\
A_{\scriptscriptstyle G_{33}^6}&=\begin{bmatrix}
0&0&0&0&1&1\\
0&0&1&1&0&0\\
0&1&0&1&0&0\\
0&1&1&0&0&1\\
1&0&0&0&0&1\\
1&0&0&1&1&0
\end{bmatrix},\quad\det=3,
&A_{\scriptscriptstyle G_{34}^6}&=\begin{bmatrix}
0&0&0&0&1&1\\
0&0&1&1&0&0\\
0&1&0&1&0&0\\
0&1&1&0&1&1\\
1&0&0&1&0&1\\
1&0&0&1&1&0
\end{bmatrix},\quad\det=4,
\\
A_{\scriptscriptstyle G_{35}^6}&=\begin{bmatrix}
0&0&0&0&1&1\\
0&0&1&1&0&0\\
0&1&0&1&1&1\\
0&1&1&0&1&1\\
1&0&1&1&0&1\\
1&0&1&1&1&0
\end{bmatrix},\quad\det=4,
&A_{\scriptscriptstyle G_{36}^6}&=\begin{bmatrix}
0&0&0&0&1&1\\
0&0&1&1&0&1\\
0&1&0&1&0&1\\
0&1&1&0&0&1\\
1&0&0&0&0&1\\
1&1&1&1&1&0
\end{bmatrix},\quad\det=7,
\\
A_{\scriptscriptstyle G_{37}^6}&=\begin{bmatrix}
0&0&0&0&1&1\\
0&0&1&1&0&1\\
0&1&0&1&0&1\\
0&1&1&0&1&0\\
1&0&0&1&0&0\\
1&1&1&0&0&0
\end{bmatrix},\quad\det=-5,
&A_{\scriptscriptstyle G_{38}^6}&=\begin{bmatrix}
0&0&0&0&1&1\\
0&0&1&1&0&1\\
0&1&0&1&0&1\\
0&1&1&0&1&1\\
1&0&0&1&0&0\\
1&1&1&1&0&0
\end{bmatrix},\quad\det=0,
\\
A_{\scriptscriptstyle G_{39}^6}&=\begin{bmatrix}
0&0&0&0&1&1\\
0&0&1&1&1&1\\
0&1&0&1&1&1\\
0&1&1&0&1&1\\
1&1&1&1&0&0\\
1&1&1&1&0&0
\end{bmatrix}\quad\text{has twins},
&A_{\scriptscriptstyle G_{40}^6}&=\begin{bmatrix}
0&0&0&0&1&1\\
0&0&1&1&1&1\\
0&1&0&1&1&1\\
0&1&1&0&1&1\\
1&1&1&1&0&1\\
1&1&1&1&1&0
\end{bmatrix},\quad\det=4,
\\
A_{\scriptscriptstyle G_{41}^6}&=\begin{bmatrix}
0&0&0&1&1&1\\
0&0&0&1&1&1\\
0&0&0&1&1&1\\
1&1&1&0&0&0\\
1&1&1&0&0&0\\
1&1&1&0&0&0
\end{bmatrix}\quad\text{has twins},
&A_{\scriptscriptstyle G_{42}^6}&=\begin{bmatrix}
0&0&0&1&1&1\\
0&0&0&1&1&1\\
0&0&0&1&1&1\\
1&1&1&0&0&0\\
1&1&1&0&0&1\\
1&1&1&0&1&0
\end{bmatrix}\quad\text{has twins},
\\
A_{\scriptscriptstyle G_{43}^6}&=\begin{bmatrix}
0&0&0&1&1&1\\
0&0&0&1&1&1\\
0&0&0&1&1&1\\
1&1&1&0&0&1\\
1&1&1&0&0&1\\
1&1&1&1&1&0
\end{bmatrix}\quad\text{has twins},
&A_{\scriptscriptstyle G_{44}^6}&=\begin{bmatrix}
0&0&0&1&1&1\\
0&0&0&1&1&1\\
0&0&0&1&1&1\\
1&1&1&0&1&1\\
1&1&1&1&0&1\\
1&1&1&1&1&0
\end{bmatrix}\quad\text{has twins},
\\\
A_{\scriptscriptstyle G_{45}^6}&=\begin{bmatrix}
0&0&0&1&1&1\\
0&0&1&0&1&1\\
0&1&0&0&1&1\\
1&0&0&0&1&1\\
1&1&1&1&0&0\\
1&1&1&1&0&0
\end{bmatrix}\quad\text{has twins},
&A_{\scriptscriptstyle G_{46}^6}&=\begin{bmatrix}
0&0&0&1&1&1\\
0&0&1&0&1&1\\
0&1&0&0&1&1\\
1&0&0&0&1&1\\
1&1&1&1&0&1\\
1&1&1&1&1&0
\end{bmatrix},\quad\det=7,
\\
A_{\scriptscriptstyle G_{47}^6}&=\begin{bmatrix}
0&0&0&1&1&1\\
0&0&1&0&1&1\\
0&1&0&1&1&1\\
1&0&1&0&1&1\\
1&1&1&1&0&0\\
1&1&1&1&0&0
\end{bmatrix}\quad\text{has twins},
&A_{\scriptscriptstyle G_{48}^6}&=\begin{bmatrix}
0&0&0&1&1&1\\
0&0&1&0&1&1\\
0&1&0&1&1&1\\
1&0&1&0&1&1\\
1&1&1&1&0&1\\
1&1&1&1&1&0
\end{bmatrix},\quad\det=3,
\\
A_{\scriptscriptstyle G_{49}^6}&=\begin{bmatrix}
0&0&0&1&1&1\\
0&0&1&1&1&1\\
0&1&0&1&1&1\\
1&1&1&0&0&0\\
1&1&1&0&0&1\\
1&1&1&0&1&0
\end{bmatrix},\quad\det=-1,
&A_{\scriptscriptstyle G_{50}^6}&=\begin{bmatrix}
0&0&0&1&1&1\\
0&0&1&1&1&1\\
0&1&0&1&1&1\\
1&1&1&0&0&1\\
1&1&1&0&0&1\\
1&1&1&1&1&0
\end{bmatrix}\quad\text{has twins},
\\
A_{\scriptscriptstyle G_{51}^6}&=\begin{bmatrix}
0&0&0&1&1&1\\
0&0&1&1&1&1\\
0&1&0&1&1&1\\
1&1&1&0&1&1\\
1&1&1&1&0&1\\
1&1&1&1&1&0
\end{bmatrix},\quad\det=3,
&A_{\scriptscriptstyle G_{52}^6}&=\begin{bmatrix}
0&0&1&1&1&1\\
0&0&1&1&1&1\\
1&1&0&0&1&1\\
1&1&0&0&1&1\\
1&1&1&1&0&0\\
1&1&1&1&0&0
\end{bmatrix}\quad\text{has twins},
\\
A_{\scriptscriptstyle G_{53}^6}&=\begin{bmatrix}
0&0&1&1&1&1\\
0&0&1&1&1&1\\
1&1&0&0&1&1\\
1&1&0&0&1&1\\
1&1&1&1&0&1\\
1&1&1&1&1&0
\end{bmatrix}\quad\text{has twins},
&A_{\scriptscriptstyle G_{54}^6}&=\begin{bmatrix}
0&0&1&1&1&1\\
0&0&1&1&1&1\\
1&1&0&1&1&1\\
1&1&1&0&1&1\\
1&1&1&1&0&1\\
1&1&1&1&1&0
\end{bmatrix}\quad\text{has twins},
\end{align*}
and $G_{55}^6=K_5$ with $\det K_5=-5$. The above list shows that only $G_{38}^6$ possesses all the required properties.

\bibliography{qms}{}
\bibliographystyle{plain}

\end{document}